\documentclass{amsart}

\usepackage{amsmath,amssymb,amsthm}

\usepackage{hyperref}

\usepackage{graphicx}

\newcommand{\dom}{\mathop{\mathrm{Dom}}}
\newcommand{\symm}[1]{\mathop{\mathrm{Symm}}\left(#1\right)}
\newcommand{\nuke}{\mathcal{N}}
\newcommand{\bdry}{\mathfrak{d}}
\newcommand{\coc}{\nu}
\newcommand{\be}{\mathsf{o}}
\newcommand{\en}{\mathsf{t}}
\newcommand{\img}[1]{\mathop{\mathrm{IMG}}\left(#1\right)}
\newcommand{\arr}{\longrightarrow}
\newcommand{\X}{\mathcal{X}}
\newcommand{\G}{\mathfrak{G}}
\newcommand{\gH}{\mathfrak{H}}
\newcommand{\cH}{\mathcal{H}}
\newcommand{\R}{\mathbb{R}}
\newcommand{\C}{\mathbb{C}}
\newcommand{\Z}{\mathbb{Z}}
\newcommand{\M}{\mathcal{M}}
\newcommand{\N}{\mathbb{N}}
\newcommand{\J}{\mathcal{J}}
\newcommand{\V}{\mathcal{V}}
\newcommand{\T}{\mathcal{T}}
\newcommand{\wt}{\widetilde}

\newtheorem{theorem}{Theorem}[section]
\newtheorem{proposition}[theorem]{Proposition}
\newtheorem{corollary}[theorem]{Corollary}
\newtheorem{lemma}[theorem]{Lemma}

\theoremstyle{definition}
\newtheorem{defi}{Definition}[section]
\newtheorem{example}{Example}[section]

\begin{document}
\title{Finitely presented groups associated with expanding maps}
\author{Volodymyr Nekrashevych}

\maketitle

\begin{abstract}
We associate with every locally expanding self-covering $f:\M\arr\M$ of a
compact path connected metric space a finitely presented group
$\V_f$. We prove that this group is a complete invariant of the
dynamical system: two groups $\V_{f_1}$ and $\V_{f_2}$ are isomorphic
as abstract groups if and only if the corresponding dynamical systems
are topologically conjugate. We also show that the commutator subgroup
of $\V_f$ is simple, and give a topological interpretation of
$\V_f/\V_f'$.
\end{abstract}

\tableofcontents

\section{Introduction}

A dynamical system is \emph{finitely presented} if it can be
represented as the factor of a shift of finite type by an equivalence
relation that is also a shift of finite type,
see~\cite{fried,curnpap:symb}. 
If we think of finite type as analogous
to finite generation (of a group or of a normal subgroup), then the
notion of a finitely presented dynamical system becomes analogous to the notion of a finitely
presented group. But the relation is deeper than just a superficial
analogy.

Condition of being finitely presented for a dynamical system is very
closely related to dynamical hyperbolicity, see~\cite{fried}. For
example, if $f:\J\arr\J$ is a locally expanding self-covering of a
compact metric space, then $f$ is finitely presented. Dynamical
hyperbolicity is very closely related to Gromov hyperbolicity for
groups, see~\cite{gro:hyperb,curnpap:symb,nek:hyperbolic}, and finite
presentation is an important property of hyperbolic groups.

The aim of this paper is to show a new connection between expanding
maps and finite presentations. We naturally associate with every
finite degree self-covering
$f:\M\arr\M$ of a path-connected space $\M$ a group $\V_f$ with the
following property (see Theorem~\ref{th:finitepresentation} and Theorem~\ref{th:classification}).

\begin{theorem}
\label{th:first}
If $f:\M\arr\M$ is a locally expanding self-covering of a compact
path connected metric space then $\V_f$ is finitely presented.

If $f_i:\M_i\arr\M_i$ are as above, then groups $\V_{f_1}$ and
$\V_{f_2}$ are isomorphic if and only if $f_1$ and $f_2$ are
topologically conjugate, i.e., there exists a homeomorphism
$\phi:\M_1\arr\M_2$ such that $f_1=\phi^{-1}\circ f_2\circ\phi$.
\end{theorem}

We also show that the commutator subgroup of $\V_f$ is simple, and
give a dynamical interpretation of the abelianization $\V_f/\V_f'$,
see Theorem~\ref{th:commutatorsimple} and
Proposition~\ref{pr:homology}.

The groups $\V_f$ are defined in the following way. Let $f:\M\arr\M$
be a finite degree covering map, and suppose that $\M$ is path
connected. Choose $t\in\M$, and consider the tree $T_t$ of preimages of $t$
under the iterations of $f$. Its set of vertices is $\bigsqcup_{n\ge
  0}f^{-n}(t)$, and a vertex $v\in f^{-n}(t)$ is connected to the
vertex $f(v)\in f^{-(n-1)}(t)$. Let $\partial T_t$ be its boundary,
which can be defined as the
inverse limit of the discrete sets $f^{-n}(t)$ with respect to the
maps $f:f^{-(n+1)}(t)\arr f^{-n}(t)$.

Let $\gamma$ be a path in $\M$ connecting a vertex $v\in f^{-n}(t)$ to
a vertex $u\in f^{-m}(t)$ of the tree $T_t$. Considering lifts of
$\gamma$ by the coverings $f^k:\M\arr\M$, we get an isomorphism
$S_\gamma:T_v\arr T_u$ between subtrees $T_v, T_u$ of $T_t$. Namely,
if $\gamma_z$ is a lift of $\gamma$ starting at $z\in f^{-k}(v)$, then
$S_\gamma(z)\in f^{-k}(u)$ is the end of $\gamma_z$.
Denote by the same symbol
$S_\gamma$ the induced homeomorphism $\partial T_v\arr\partial T_u$ of
the boundaries of the subtrees, seen as clopen subsets of $\partial T_t$.

\begin{defi}
The group $\V_f$ is the group of all homeomorphisms $\partial
T_t\arr\partial T_t$ locally equal to homeomorphisms of the form
$S_\gamma:\partial T_v\arr\partial T_u$.
\end{defi}

The group $\V_f$ is generated by the Higman-Thompson group $G_{\deg f,
1}$ (see~\cite{hgthomp}) acting on $\partial T_t$ and the \emph{iterated monodromy group}
$\img{f}$ of $f:\M\arr\M$. The iterated monodromy group can be defined
as the subgroup of $\V_f$ consisting of homeomorphism
$S_\gamma:\partial T_v\arr\partial T_v$,
where $\gamma$ is a loop starting and ending at the basepoint $t$.
It is an
invariant of the topological conjugacy class of $f:\M\arr\M$, and it
becomes a complete invariant (in the expanding case), if we consider
it as a \emph{self-similar group}. Self-similarity is an additional
structure on a group, and it can be defined using one of several
equivalent approaches: virtual endomorphisms, wreath recursions,
bisets, or structures of an automaton group. The fact that self-similar
iterated monodromy group is a complete invariant of an expanding
self-covering is one of the main topics of~\cite{nek:book}.

From the point of view of group theory, $\V_f$ seems to be a
``cleaner'' object, since no additional structure is needed to make
it a complete invariant of a dynamical system. Besides, it is finitely
presented, unlike the iterated monodromy groups, which are typically
infinite presented. 
However, iterated monodromy groups have better functorial properties
than the groups $\V_f$, see~\cite{nek:filling}.

A group $\V_G$, analogous to $\V_f$, can be defined for any self-similar
group $G$, so that $\V_f=\V_{\img{f}}$. Groups of this type were for
the first time studied by C.~R\"over~\cite{roever,roever:comm}. In particular, he showed
that if $G$ is the Grigorchuk group~\cite{grigorchuk:80_en}, then $\V_G$ is finitely
presented, simple, and is isomorphic to the abstract commensurator of
the Grigorchuk group. The case of a general self-similar group $G$ was
studied later in~\cite{nek:bim}.

Several natural questions arise in connection with
Theorem~\ref{th:first}. For example, is the isomorphism problem
solvable for the groups $\V_f$? Equivalently, is the topological
conjugacy problem for expanding maps algorithmically solvable? 
Note that expanding maps can
be given in different ways by a finite amount of information: using
finite presentations in the sense of~\cite{fried}, using combinatorial
models in the sense of~\cite{ishiismillie,nek:models}, using iterated monodromy
groups, e.t.c..

Another natural question is whether the groups $\V_f$, similarly to the
Higman-Thompson groups (see~\cite{brown:finiteness}), 
satisfy the finiteness condition $F_\infty$,
i.e., if they have classifying spaces with finite $n$-dimensional
skeleta for all $n$.
It would be also interesting to study homology of $\V_f$ in relation with
homological properties of the dynamical system.

The structure of the paper is as follows. In 
``Definition of the groups $\V_f$'' we give a review of terminology
related to
rooted trees, and define the groups $\V_f$. In the next section
``Symbolic coding'' we encode the vertices of the tree of preimages
$T_t$ by finite words over an alphabet $X$, and show that $\V_f$
contains a natural copy of the Higman-Thompson group, and that $\V_f$
is generated by the Higman-Thompson group and the iterated monodromy
group. We also give a review of the basic notions of the theory of
self-similar groups, and define the groups $\V_G$ associated with
self-similar groups, following~\cite{nek:bim}.

In Section~4 we prove that the commutator subgroup $\V_G'$ of $\V_G$
is simple for any self-similar group $G$. In particular, $\V_f'$ is
simple for any map $f$.
Note that the fact that every proper quotient of $\V_G$ is abelian,
i.e., that every non-trivial normal subgroup of $\V_G$ contains
$\V_G'$ was already proved in~\cite{nek:bim}, and we use this fact in
our proof. Later, in the next section we give an interpretation of $\V_f/\V_f'$ in
topological terms. Namely, we prove the following (see Proposition~\ref{pr:homology}).

\begin{proposition}
Suppose that $f:\M_1\arr\M$ is expanding, $\M$ is path-connected and
semi-locally simply connected, and $\M_1\subseteq\M$. 

If $\deg f$ is even, then $\V_f/\V_f'$ is isomorphic to the quotient
of $H_1(\M)$ by the range of the endomorphism $1-\iota_*\circ f^!$.

If $\deg f$ is odd, then $\V_f/\V_f'$ is isomorphic to the quotient of
$\Z/2\Z\oplus H_1(\M)$ by the range of the endomorphism $1-\sigma_1$,
where $\sigma_1(t, c)=(t+\mathop{\mathrm{sign}}(c), \iota_*\circ f^!(c))$.
\end{proposition}

Here $\iota_*:H_1(\M_1)\arr H_1(\M)$ is the homomorphism induced by
the identical embedding $\iota:\M_1\arr\M$, the homomorphism $f^!:H_1(\M)\arr
H_1(\M_1)$ maps a cycle $c$ to its full preimage $f^{-1}(c)$, and
$\mathop{\mathrm{sign}}:H_1(\M)\arr\Z/2\Z$ maps a cycle $c$ defined by
a loop $\gamma$ to $1$ if $\gamma$ acts as an odd permutation on the
fiber of $f$, and to $0$ otherwise.

The main result of Section~5 is existence of finite presentation of
$\V_f$ when $f$ is expanding. More generally, we show that $\V_G$ is
finitely presented, if $G$ is a \emph{contracting self-similar group},
see Theorem~\ref{th:finitepresentation}. We also give (in
Subsection~5.2) a general
definition of the groups $\V_f$ for expanding maps $f:\M\arr\M$ (where
$\M$ is not required to be path connected). 

Section ``Dynamical systems and groupoids'' is an overview of the
theory of limit dynamical systems of contracting self-similar groups
and basic results of hyperbolic groupoids, following~\cite{nek:book}
and~\cite{nek:hyperbolic}. These results are needed for the proof the
fact that $\V_f$ is a complete invariant of the dynamical system in
the expanding case, which is proved (Theorem~\ref{th:classification})
in the last section of the paper.

The general scheme of the proof of Theorem~\ref{th:classification}
is as follows. First, we show, using a theorem of M.~Rubin~\cite{rubin:reconstr},
that two groups $\V_{f_1}$ and $\V_{f_2}$ are isomorphic if and only
if their actions on the corresponding boundaries of trees are
topologically conjugate. This implies, that the groupoid of germs $\G$
of
the action of the group $\V_f$ on the boundary of the tree is uniquely
determined by the group $\V_f$.

The groupoid $\G$ is hyperbolic, and hence it uniquely determines the
equivalence class of its dual (see~\cite{nek:hyperbolic}), which is
the groupoid generated by the germs of $f:\M\arr\M$. It remains to
show that the dynamical system $f:\M\arr\M$ is uniquely determined (up
to topological conjugacy) by the equivalence class of the groupoid of
germs generated by it. This is proved using the techniques of
hyperbolic groupoids. Connectedness of $\M$ is used in the proof in an
essential way.

It is a natural question to ask if Theorem~\ref{th:classification} is
true in general (without the condition that $\M$ is path connected). 

As a corollary of the proof of Theorem~\ref{th:classification}, we
clarify the relation between the groupoid-theoretic equivalence of
groupoids of germs associated with expanding dynamical systems, and
their topological conjugacy, see Theorem~\ref{th:groupoidequivalence}.

\subsection*{Acknowledgments}

I am grateful to the organizers of the LMS Durham Symposium
``Geometric and Cohomological Group Theory'' for inviting me to give a
talk, which inspired me to return to the topics of this paper.

The paper is based in part on work supported by NSF grant DMS1006280.

\section{Definition of the groups $\V_f$}

\subsection{Rooted trees}

Let $T$ be a locally finite rooted tree, and let $v, u$ be its
vertices. We write $v\preceq u$ if the path connecting the
root to $u$ passes through $v$. This defines a partial order on the
set of vertices of $T$, and $T$ is its Hasse diagram (though we tend
to draw rooted trees ``upside down'' with the root on top).

We denote by $T_v$ the sub-tree with root $v$ spanned by all vertices $u$ such that
$v\preceq u$. We have $v\preceq u$ if and only if $T_v\supseteq
T_u$. If $v$ and $u$ are incomparable, then $T_v$ and $T_u$
are disjoint.

\begin{figure}
\centering
\includegraphics{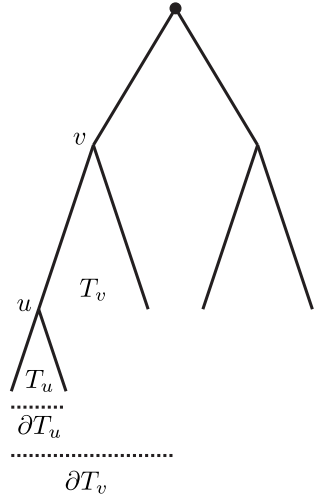}
\caption{Rooted tree}
\label{fig:tree}
\end{figure}

\emph{Boundary} $\partial T$ of the tree $T$ is the set of all
infinite simple paths starting at the root of $T$. The boundary
$\partial T_v$ is naturally identified with the set of paths $w\in\partial T$
passing through $v$. The sets $\partial T_v$ form a basis of open sets
of a natural topology on $\partial T$. The subsets $\partial T_v$ are clopen (closed and
open), and every clopen subset of $\partial T_v$ is disjoint union
of a finite number of sets of the form $\partial T_v$.

The \emph{$n$th level} of $T$ is the set of vertices that are on
distance $n$ from the root of the tree. 
An \emph{antichain} of a rooted tree $T$ is a set of
pairwise incomparable vertices. A finite antichain $A$ is said to be
\emph{complete} if it is maximal, i.e., if 
every set $B$ of vertices of $T$ properly containing $A$
is not an antichain. For example, every level of $T$ is a complete antichain.

A set $A$ is an antichain if and only if the sets $\partial T_v$ for
$v\in A$ are disjoint. It follows that $A$ is a complete antichain if
and only if $\partial T$ is disjoint union of the sets $\partial T_v$
for $v\in A$.

Let $X$ be a finite set. We denote by $X^*$ the free monoid generated
by $X$, i.e., the set of all finite words $x_1x_2\ldots x_n$ for
$x_i\in X$, together with the empty word $\varnothing$.
\emph{Length} of a word $v\in X^*$ is the number of its letters.

We introduce on the set $X^*$ structure of a rooted tree coinciding with
the left Cayley graph of the free monoid. Namely, two finite words are
connected by an edge if and only if they are of the form $vx$ and $v$
for $v\in X^*$ and $x\in X$. The empty word is the root of the tree
$X^*$.
We have $v\preceq u$ for $v, u\in X^*$ if and only if $v$ is a
beginning of $u$.
The subtree $T_v$ of $T=X^*$ for $v\in X^*$ is the
set $vX^*$ of all words starting with $v$. The $n$th level of $X^*$ is
the set $X^n$ of words of length $n$.

The boundary of the tree $X^*$ is naturally identified with the space
$X^\omega$ of right-infinite sequences $x_1x_2\ldots$ of elements of
$X$. The topology on the boundary coincides with the direct product
topology on $X^\omega$.

\subsection{Definition}

Let $\M$ be a topological space. A \emph{partial self-covering} is a finite degree covering map
$f:\M_1\arr\M$, where $\M_1\subseteq\M$. 

If $f:\M_1\arr\M$ is a partial self-covering, then we can iterated it
as a partial self-map of $\M$. Then the $n$th iteration
$f^n:\M_n\arr\M$ is also a partial self-covering. Here
$\M_n\subset\M_{n-1}\subset\ldots\subset\M_1$ are domains of the
iterations of $f$, defined by the condition $\M_{n+1}=f^{-1}(\M_n)$.

For a point $t\in\M$ denote by $T_t$ the \emph{tree of preimages} of
$t$ under the iterations of $f$, i.e., the tree with the set of
vertices equal to the formal disjoint union $\bigsqcup_{n\ge
  0}f^{-n}(t)$ of the sets of preimages of $t$ under the iterations
$f^n:\M_n\arr\M$. Here $f^{-0}(t)=\{t\}$ consists of the root of the
tree, and a vertex $v\in f^{-n}(t)$ is connected by an edge to the
vertex $f(v)\in f^{-(n-1)}(t)$, see Figure~\ref{fig:prtree}. 
If $v$ is a vertex of $T_t$, then the tree of preimages $T_v$ is in a
natural way a sub-tree of the tree $T_t$, and our notation agrees with
the notation of the previous subsection. 

\begin{figure}
\centering\includegraphics{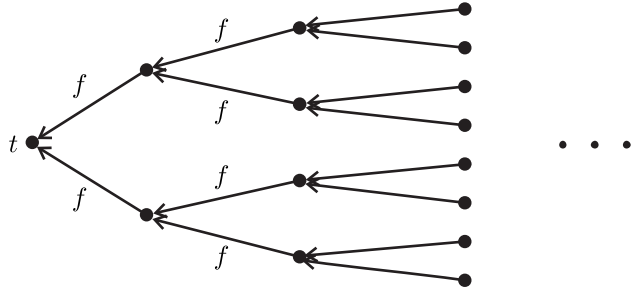}
\caption{Tree $T_t$}
\label{fig:prtree}
\end{figure}

Assume now that $\M$ is path connected. 
Let $t_1, t_2\in\M$, and let $\gamma$ be a path from $t_1$ to $t_2$ in
$\M$. Then for every $n\ge 0$ and every $v\in f^{-n}(t_1)$ there is a
unique lift by the covering $f^n:\M_n\arr\M$ of $\gamma$ starting in $v$. Let
$S_\gamma(v)\in f^{-n}(t_2)$ be its end. It is easy to see that
the map $S_\gamma:T_{t_1}\arr T_{t_2}$ is an isomorphism of the rooted
trees, see Figure~\ref{fig:sgamma}. It defines a homeomorphism of their boundaries
$S_\gamma:\partial T_{t_1}\arr\partial T_{t_2}$, which we will denote by the
same letter.

\begin{figure}
\centering\includegraphics{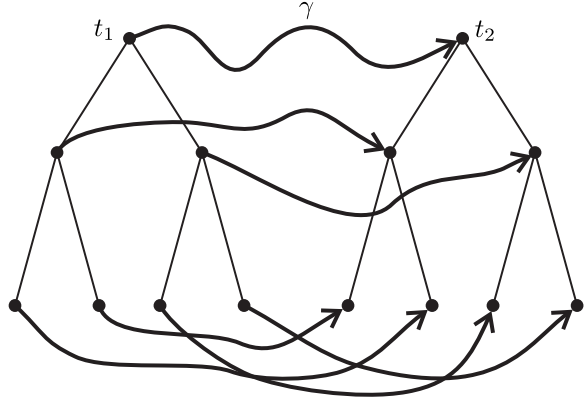}
\caption{Isomorphism $S_\gamma$}
\label{fig:sgamma}
\end{figure}

\begin{defi} Let $f:\M_1\arr\M$ be a partial self-covering, and let $t\in\M$.
Denote by $\T_f$ the semigroup of partial homeomorphisms of  $\partial T_t$
generated by the homeomorphisms of the form $S_\gamma:\partial T_{v_1}\arr
\partial T_{v_2}$, where $\gamma$ is a
path connecting points $v_1, v_2\in\bigcup_{n\ge 0}f^{-n}(t)$.
\end{defi}

The semigroup $\T_f$ contains the \emph{zero} map between empty
subsets of $T_t$. A product $S_{\gamma_1}S_{\gamma_2}$ is zero if the
range of $S_{\gamma_2}$ is disjoint from the domain of
$S_{\gamma_1}$.

Let $S_{\gamma_1}:\partial T_{v_1}\arr\partial T_{v_2}$ and
$S_{\gamma_2}:\partial T_{u_1}\arr
\partial T_{u_2}$ be two generators of $\T_f$. The product
$S_{\gamma_1}S_{\gamma_2}$ is non-zero if and only if $T_{u_2}$ and
$T_{v_1}$ are not disjoint, i.e., if either  $T_{u_2}\supseteq
T_{v_1}$, or $T_{u_2}\subseteq
T_{v_1}$. It the first case, $v_1$ is a
preimage of $u_2$ under some iteration $f^k$ of $f$. Let $\gamma_2'$
be the unique lift of $\gamma_2$ by $f^k$ that ends in $v_1$. Then it
follows from the definition of the transformations $S_\gamma$
that
\begin{equation}\label{eq:gamma1}
  S_{\gamma_1}S_{\gamma_2}=S_{\gamma_1\gamma_2'},\end{equation}
see Figure~\ref{fig:sg12}. Here and in the sequel, we multiply paths
as we compose functions: in the product $\gamma_1\gamma_2'$ the path
$\gamma_2'$ is passed before the path $\gamma_1$.

Similarly, if $T_{u_2}\subseteq T_{v_1}$, then $u_2$ is a
$f^k$-preimage of $v_1$ for some $k\ge 0$, and
\begin{equation}\label{eq:gamma2}S_{\gamma_1}S_{\gamma_2}=S_{\gamma_1'\gamma_2},\end{equation}
where $\gamma_1'$ is the lift of $\gamma_1$ by $f^k$ starting in
$u_2$.

\begin{figure}
\centering\includegraphics{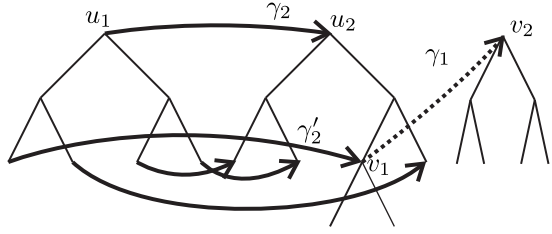}
\caption{Composition $S_{\gamma_1}S_{\gamma_2}$}
\label{fig:sg12}
\end{figure}

It follows that all non-zero elements of $\T_f$ are of the form
$S_\gamma$ for some path $\gamma$ in $\M$ connecting two vertices of $T_t$. Note
also that $\T_t$ is an inverse semigroup, where
$S_\gamma^*=S_{\gamma^{-1}}$.

Let $A_1, A_2$ be two complete antichains of $T_t$ of equal cardinality. Choose a
bijection $\alpha:A_1\arr A_2$ and a collection of paths $\gamma_a$
from $a\in A_1$ to the corresponding vertex
$\alpha(a)$, see Figure~\ref{fig:elements}.
Let $g:\partial T_t\arr\partial T_t$ be the map given by the rule
\[g(w)=S_{\gamma_v}(w),\qquad\text{if $w\in \partial T_v$.}\]
It is easy to see that $g:\partial T_t\arr\partial T_t$ is a
homeomorphism. Denote by $\V_f$ the set of all such homeomorphisms.

\begin{figure}
\centering
\includegraphics{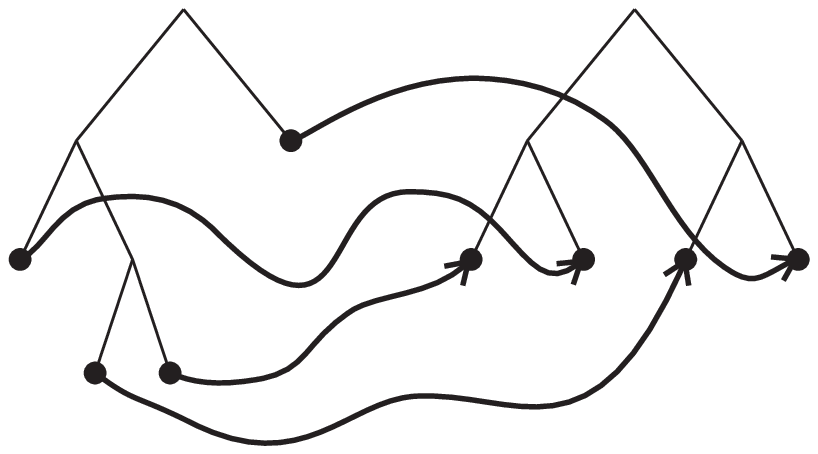}
\caption{Elements of $\V_f$}
\label{fig:elements}
\end{figure}

We will represent homeomorphisms $g\in\V_f$ by tables of the form
\begin{equation}\label{eq:table}\left(\begin{array}{cccc}v_1 & v_2 & \ldots & v_n\\ \gamma_{v_1} &
    \gamma_{v_2} & \ldots & \gamma_{v_n}\\ \alpha(v_1) & \alpha(v_2) &
    \ldots & \alpha(v_n)\end{array}\right),\end{equation}
where the first row is the list of vertices of a complete antichain,
and $g(w)=S_{\gamma_{v_i}}(w)$ for all $w\in\partial T_{v_i}$.

An \emph{elementary splitting} of such a table is the operation of replacing a column
$\left(\begin{array}{c} u \\ \gamma_u \\ \alpha(u)\end{array}\right)$
by the array $\left(\begin{array}{cccc}u_1 & u_2 & \ldots & u_d\\ \gamma_{u_1} &
    \gamma_{u_2} & \ldots & \gamma_{u_d}\\ w_1 & w_2 &
    \ldots & w_d\end{array}\right)$, where $\{u_1, u_2, \ldots,
u_d\}=f^{-1}(u)$, $\gamma_{u_i}$ is the lift of $\gamma_u$ by $f$
starting at $u_i$, and $w_i$ is the end of $\gamma_{u_i}$, see
Figure~\ref{fig:splitting}.
A \emph{splitting} of a table is the results of a finite sequence of
elementary splittings.

\begin{figure}
\centering
\includegraphics{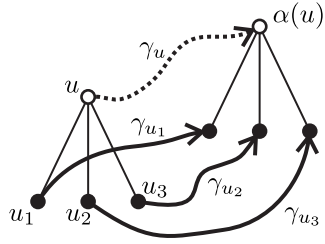}
\caption{Elementary splitting}
\label{fig:splitting}
\end{figure}

It follows directly from the definition that splitting of a table does
not change the homeomorphism $g\in\V_f$ that it defines.
It is obvious that if $g_1$ and
$g_2$ are defined by tables of the form
\[\left(\begin{array}{cccc}v_1 & v_2 & \ldots & v_n\\ \gamma_{v_1} &
    \gamma_{v_2} & \ldots & \gamma_{v_n}\\ u_1 & u_2 & \ldots &
    u_n\end{array}\right)\]
and
\[\left(\begin{array}{cccc}w_1 & w_2 & \ldots & w_n\\ \gamma_{w_1} &
    \gamma_{w_2} & \ldots & \gamma_{w_n}\\ v_1 & v_2 & \ldots &
    v_n\end{array}\right),\]
then the composition $g_1g_2$ is defined by the table
\[\left(\begin{array}{cccc}w_1 & w_2 & \ldots & w_n\\ \gamma_{v_1}\gamma_{w_1} &
    \gamma_{v_2}\gamma_{w_2} & \ldots & \gamma_{v_n}\gamma_{w_n}\\ u_1 & u_2 & \ldots &
    u_n\end{array}\right).\]

Let \[\left(\begin{array}{cccc}a_1 & a_2 & \ldots & a_n\\ \gamma_{a_1} &
    \gamma_{a_2} & \ldots & \gamma_{a_n}\\ b_1 & b_2 & \ldots &
    b_n\end{array}\right),\quad \left(\begin{array}{cccc}c_1 & c_2 & \ldots & c_m\\ \gamma_{c_1} &
    \gamma_{c_2} & \ldots & \gamma_{c_m}\\ a_1' & a_2' & \ldots &
    a_m'\end{array}\right)\] be tables defining elements of $\V_f$. We can find a
complete antichain $A$ such that for every $v\in A$ the subtree $T_v$
is contained in a subtree $T_{a_i}$ and a subtree $T_{a_j'}$ for some
$i$ and $j$. For
example, we can take $A$ to be equal to the $k$th level of the tree
$T_t$ for $k$ big enough. Then
there exists a splitting of the first table such that its first row is
$A$, and there exists a splitting of the second table such that its
last row is $A$.

It follows that if $g_1, g_2\in\V_f$, then their composition $g_1g_2$
also belongs to $\V_f$. As a corollary, we get the following proposition.

\begin{proposition}
The set $\V_f$ is a group.
\end{proposition}

We will use sometimes instead of tables the following notation.
If $F$ and $G$ are two partial transformations with disjoint domains,
then we denote by $F+G$ their union, i.e., the map equal to $F$ on the
domain of $F$ and equal to $G$ on the domain of $G$. Then the
transformation defined by a table
\[\left(\begin{array}{cccc}v_1 & v_2 & \ldots & v_n\\ \gamma_1 &
    \gamma_2 & \ldots & \gamma_n\\ u_1 & u_2 & \ldots &
    u_n\end{array}\right)\]
is written $S_{\gamma_1}+S_{\gamma_2}+\cdots+S_{\gamma_n}$.

An elementary splitting of a table is equivalent then to application
of the identity
\[S_\gamma=\sum_{\delta\in f^{-1}(\gamma)}S_\delta,\] where
$f^{-1}(\gamma)$ is the set of all lifts of $\gamma$ by $f$.

\section{Symbolic coding}
\subsection{Two trees}
\label{ss:2trees}

Let $\{t_1, t_2, \ldots, t_d\}=f^{-1}(t)$, and choose paths $\ell_i$ in $\M$
from $t$ to $t_i$. Then $S_{\ell_i}:T_t\arr T_{t_i}$ is an
isomorphism. The elements $S_{\ell_i}\in\mathcal{T}_f$ satisfy the
relations:
\[S_{\ell_i}^*S_{\ell_i}=1,\qquad\sum_{i=1}^dS_{\ell_i}S_{\ell_i}^*=1.\]

The $C^*$-algebra defined by such relations is called the \emph{Cuntz
algebra}~\cite{cuntz}. If we denote by $S$ the row
$(S_{\ell_1}, S_{\ell_2}, \ldots, S_{\ell_d})$, and by $S^*$ the
column $(S_{\ell_1}^*, S_{\ell_2}^*, \ldots,
S_{\ell_d}^*)^\top$, then the relations can be written as matrix equalities
\[S^*S=I_d,\qquad SS^*=I_1,\]
where $I_n$ denotes the $n\times n$ identity matrix. Rings given by
these and similar defining relations were studied by
W.~Leavitt~\cite{leavitt:moduleswords,leavitt:duke}.

Let $\Gamma_t$ be the graph with the set of vertices equal to the set of
vertices of $T_t$ in which two vertices $v\in f^{-n}(t)$ and $v\in
f^{-(n+1)}(t)$  are connected by an edge if
and only if they are connected by a path equal to a lift of a path
$\ell_i$ by the covering $f^n$. In other words, the
graph $\Gamma_t$ is obtained by taking preimages of the paths $\ell_i$
under all iterations of $f$, see Figure~\ref{fig:lambda}.

\begin{figure}
\centering
\includegraphics{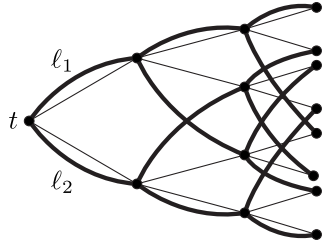}
\caption{Trees $\Gamma_t$ and $T_t$}
\label{fig:lambda}
\end{figure}

It is easy to see that $\Gamma_t$ is a tree, and its $n$th level is
equal to the $n$th level of the tree $T_t$.
It follows that any two vertices of $T_t$ are connected by a unique
simple path in $\Gamma_t$.

Let $i_{n-1}\ldots i_1i_0\in\{1, 2, \ldots, d\}^n$, and consider the
product $S_{i_{n-1}}\cdots S_{i_1}S_{i_0}\in\mathcal{T}_f$. 
According to the multiplication rule~\eqref{eq:gamma2}, it is equal to
$S_{\gamma_{n-1}\ldots\gamma_1\gamma_0}$, where $\gamma_0=\ell_{i_0}$, and
$\gamma_k$ is the lift of $\ell_k$ by $f^k$ starting at the end of
$\gamma_{k-1}$. Denote by $\Lambda(i_{n-1}\ldots i_1i_0)$ the end of
the last path $\gamma_{n-1}$. Then $S_{i_{n-1}}\cdots
S_{i_1}S_{i_0}=S_{\gamma_{i_{n-1}\ldots i_2i_1}}$, where
$\gamma_{i_{n-1}\ldots i_2i_1}$ is unique simple path in $\Gamma_t$
starting at the root and ending in $\Lambda(i_{n-1}\ldots i_1i_0)$.
The path $\gamma_{i_{n-1}\ldots i_2i_1}$ and its end
$\Lambda(i_{n-1}\ldots i_1i_0)$ satisfy the
recurrent rule:
\[\gamma_{i_{n-1}\ldots i_2i_1}=\ell_{\Lambda(i_{n-2}\ldots i_1i_0),
  i_{n-1}}\gamma_{i_{n-2}\ldots i_2i_1},\]
where $\ell_{\Lambda(i_{n-2}\ldots i_1i_0), i_{n-1}}$ is the lift of
$\ell_{i_{n-1}}$ by $f^{n-1}$ starting at $\Lambda(i_{n-2}\ldots
i_1i_0)$ (and hence ending in $\Lambda(i_{n-1}\ldots i_1i_0)$).

The map $\Lambda$ is a bijection between $\{1, 2, \ldots, d\}^n$ and the
$n$th level $f^{-n}(t)$ of the trees $T_t$ and $\Gamma_t$. It
follows directly from the description of the path
$\gamma_{i_{n-1}\ldots i_1i_0}$ that
$\Lambda(i_{n-1}\ldots i_1i_0)$ is adjacent to $\Lambda(i_{n-1}\ldots
i_2i_1)$ in $T_t$ and to $\Lambda(i_{n-2}\ldots i_1i_1)$ in $\Gamma_t$.
In other words, $T_t$ and $\Gamma_t$ are identified by $\Lambda$ with
the right and the left Cayley graphs of the free monoid generated by
$X=\{1, 2, \ldots, d\}$, respectively.

For any two sequences $i_1i_2\ldots i_n$ and $j_1j_2\ldots j_m\in
X^*$, the product \[S_{\ell_{i_1}}S_{\ell_{i_2}}\cdots
S_{\ell_{i_n}}(S_{j_1}S_{j_2}\ldots S_{j_m})^*\] is
equal to $S_\gamma$, where $\gamma$ is the path inside $\Gamma_t$
from $\Lambda(j_1j_2\ldots j_m)$ to $\Lambda(i_1i_2\ldots i_n)$. This
follows directly from the definitions of $\Gamma_t$, $\Lambda$, and
rules~\eqref{eq:gamma1} and~\eqref{eq:gamma2}.

We will use notation $S_x=\Lambda^{-1}S_{\ell_x}\Lambda$ and $S_{x_1x_2\ldots
  x_n}=S_{x_1}S_{x_2}\cdots S_{x_n}$, for $x, x_i\in X$. Then, by the definition of
$\Lambda$, the transformations  $S_{x_1x_2\ldots x_n}$ of $X^\omega$ are given by the rule
\[S_{x_1x_2\ldots x_n}(v)=x_1x_2\ldots x_nv.\]
For $v, u\in X^*$, the transformation $S_vS_u^*$ is defined on
$uX^\omega$, and acts by the rule
\[S_vS_u^*(vw)=uw.\]
In particular, $\sum_{x\in X}S_xS_x^*=1$, and we obviously have $S_x^*S_x=1$.

\subsection{The Higman-Thompson group}

Let $A_1$ and $A_2$ be complete antichains in $X^*$, and let
$\alpha:A_1\arr A_2$ be a bijection. Then $g_\alpha=\sum_{v\in
  A_1}S_{\alpha(v)}S_v^*$ is a homeomorphism of $X^\omega$ defined
by the rule
\[g_\alpha(vw)=\alpha(v)w,\]
for all $v\in A_1$ and $w\in X^\omega$. The set of such homeomorphisms
$g_\alpha$ is the \emph{Higman-Thompson group} group $G_{|X|, 1}$, see~\cite{hgthomp}, which we
will denote by $\V_X$ of $\V_d$, where $d=|X|$. 

Its copy $\Lambda\cdot\V_X\cdot\Lambda^{-1}$ in $\V_f$ is the
group defined by the paths in the tree $\Gamma_t$. Namely, for any bijection
$\alpha:A_1\arr A_2$ between complete antichains of $T_t$ there exist
unique simple paths $\gamma_v$ connecting $v\in A_1$ to
$\alpha(v)\in A_2$ inside the tree $\Gamma_t$. Then the corresponding
element of $\V_f$ is equal to $\sum_{v\in A_1}S_{\gamma_v}$.

The following simple lemma will be useful later (for a proof, see, for
example~\cite[Lemma~9.12]{nek:bim}).

\begin{lemma}
\label{lem:incomplete}
Let $A_1, A_2\subset X^*$ be finite incomplete (i.e., non-maximal)
antichains, and let $\alpha:A_1\arr A_2$ be a bijection. Then there
exists $g\in\V_X$ such that $g(vw)=\alpha(v)w$ for all $v\in A_1$ and
$w\in X^\omega$.
\end{lemma}

\subsection{The iterated monodromy group}

Every element $\gamma$ of the fundamental group $\pi_1(\M, t)$ 
defines an element $S_\gamma:\partial T_t\arr\partial T_t$ of $\V_f$. We get in
this way a natural homomorphism $\gamma\mapsto S_\gamma$ from
$\pi_1(\M, t)$ to $\V_f$. Its image is called the \emph{iterated
  monodromy group} of $f$ and is denoted $\img{f}$. It acts on $T_t$
by automorphisms, so that the action on the $n$th level coincides with
the natural \emph{monodromy action} associated with the covering
$f^n:\M_n\arr\M$, see~\cite[Chapter~5]{nek:book},~\cite{bgn,nek:bath}.

Let us choose paths $\ell_i$ connecting the root $t$ to the vertices
of the first level $f^{-1}(t)$ of the tree $T_t$. Let $\Gamma_t$ be the
tree obtained by taking lifts of the paths $\ell_i$ by iterations of
$f$, as in Subsection~\ref{ss:2trees}.

For a vertex $v$ of $T_t$, denote by $\ell_v$ the unique simple path
inside $\Gamma_t$ from $t$ to $v$.
Then for an arbitrary path $\gamma$ in $\M$ starting in a vertex $v$ and ending in
a vertex $u$ of $T_t$, the path $\ell_u^{-1}\gamma\ell_v$ is a loop
based at $t$. Let $g=S_{\ell_u^{-1}\gamma\ell_v}$ be the corresponding
element of $\img{f}$. Then we have
\[S_\gamma=S_{\ell_u}S_{\ell_u^{-1}\gamma\ell_v}S_{\ell_v}^*=S_{\ell_u}gS_{\ell_v}^*.\]
Hence we get the following description of the elements of $\V_f$.
\begin{lemma}
\label{lem:imginside}
Let $g\in\V_f$ be defined by a table 
$\left(\begin{array}{cccc}
   v_1 & v_2 & \ldots & v_n\\
   \gamma_1     & \gamma_2     & \ldots & \gamma_n \\
   u_1 & u_2 & \ldots & u_n
\end{array}\right)$. Denote
$g_i=S_{\ell_{u_i}}^*S_{\gamma_i}S_{\ell_{v_i}}$. Then
$g_i\in\img{f}$, and
$g=\sum_{i=1}^n S_{\ell_{u_i}}g_iS_{\ell_{v_i}}^*$.
\end{lemma}

Let $g\in\img{f}$ be defined by a loop $\gamma$, and let
$x\in f^{-1}(t)$ be a vertex of the first level. Let $\gamma_x$ be the
lift of $\gamma$ by $f$ starting at $x$, and let $y$ be its end. Then
we have
\begin{equation}\label{eq:ssimilarity}
gS_{\ell_x}=S_\gamma S_{\ell_x}=S_{\gamma_x\ell_x}=
S_{\ell_y\ell_y^{-1}\gamma_x\ell_x}=S_{\ell_y}S_{\ell_y^{-1}\gamma_x\ell_x}.
\end{equation}
Note that $\ell_y^{-1}\gamma_x\ell_x$ is a loop based at $t$, i.e., an
element of $\pi_1(\M, t)$, see Figure~\ref{fig:recurn}. 

\begin{figure}
\centering
\includegraphics{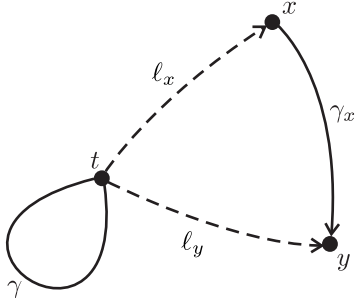}
\caption{Iterated monodromy recursion}
\label{fig:recurn}
\end{figure}

Let us conjugate the action of $\V_f$ on $\partial T_t$ (and the
action of $\img{f}$ on $T_t$) to an action on $X^\omega$ (and $X^*$)
using the isomorphism $\Lambda:X^*\arr T_t$. We call the obtained
actions of $\V_f$ and $\img{f}$ \emph{standard}. Then
formula~\eqref{eq:ssimilarity}
proves the following lemma, see also~\cite[Proposition~5.2.2]{nek:book}.

\begin{lemma}
\label{lem:ssimilarity}
For every $g\in\img{f}$ and every $x\in f^{-1}$ there exist
$h\in\img{f}$ and $y\in f^{-1}$ such that
\[gS_x=S_yh.\]
Moreover, if $g$ is defined by a loop $\gamma$, then $h$ is defined by
the loop $\ell_y^{-1}\gamma_x\ell_x$, where $\gamma_x$ is the lift of
$\gamma$ by $f$ starting at $x$.
\end{lemma}

If $gS_x=S_yh$ for $g, h\in\img{f}$ and $x, y\in X$, then we
denote $h=g|_x$ and $y=g(x)$. Note that the last equality agrees with
the definition of $g$ as an automorphism $S_\gamma$ of the tree $T_t$.

Since $1=\sum_{x\in X}S_xS_x^*$, we have
\begin{equation}
\label{eq:gsplitting}
g=\sum_{x\in X}gS_xS_x^*=\sum_{x\in X}S_{g(x)}g|_xS_x^*,
\end{equation}
which gives us a splitting rule for the expressions for elements of
$\V_f$ given in Lemma~\ref{lem:imginside}.

Namely, we get the following description of $\V_f$ in terms of
$\img{f}$ and formula~\eqref{eq:gsplitting}.

\begin{proposition}
\label{pr:symbolic}
The group $\V_f$ is isomorphic to homeomorphisms of $X^\omega$ of the form
\begin{equation}\label{eq:SugSv}
\sum_{v\in A_1}S_{\alpha(v)}g_vS_v^*,
\end{equation}
where $g_v\in\img{f}$, $A_1$ is a complete antichain in $X^*$, and
$\alpha:A_1\arr A_2$ is a bijection of $A_1$ with a complete antichain
$A_2$. Two elements of $\V_f$ given by expressions of the form~\eqref{eq:SugSv}
are equal if and only if they can be made equal after repeated
applications of the splitting rules~\eqref{eq:gsplitting} to the elements
$g_v$.
\end{proposition}

Equivalently, we can use the table notation, and represent the
element~\eqref{eq:SugSv} by the table
\begin{equation}
\label{eq:tablesG}
\left(\begin{array}{cccc} v_1 & v_2 & \ldots & v_m\\ g_{v_1} &
    g_{v_2} & \ldots & g_{v_m}\\ \alpha(v_1) & \alpha(v_2) & \ldots &
    \alpha(v_m)\end{array}\right),
\end{equation}
where the splitting rule is the operation of replacing a column by the
array
\begin{equation}
\label{eq:tablesspittingG}
\left(\begin{array}{c}v \\ g\\ u\end{array}\right)\mapsto\left(\begin{array}{cccc}vx_1 & vx_2 & \ldots & vx_d\\ g|_{x_1} &
    g|_{x_2} & \ldots & g|_{x_d}\\ ug(x_1) & ug(x_2) & \ldots &
    ug(x_d)\end{array}\right).
\end{equation}

\begin{example}
\label{ex:addingmachine}
Consider the self-covering $f:x\mapsto 2x$ of the circle $\R/\Z$. Take
$t=0$ as the basepoint. Its preimages are $0$ and $1/2$. Let $\ell_0$
be the trivial path at $0$, and let $\ell_1$ be the path from $0$ to
$1/2$ equal to the image of the segment $[0, 1/2]\subset\R$. Let $\gamma$
be the generator of $\pi_1(\R/\Z, 0)$ equal to the image of the
segment $[0, 1]\subset\R$ with the natural (increasing on $[0, 1]$)
orientation. It has two lifts by the covering $f$: \[\gamma_0=[0,
1/2],\qquad\gamma_1=[1/2, 1].\]
Note that $\gamma_0=\ell_1$.

By~\eqref{eq:gamma2}, 
\begin{equation}
\label{eq:admach1}
S_\gamma S_{\ell_0}=S_{\gamma_0\ell_0}=S_{\ell_1},
\end{equation}
and
\begin{equation}
\label{eq:admach2}
S_\gamma
S_{\ell_1}=S_{\gamma_1\ell_1}=S_{\ell_0\gamma}=S_{\ell_0}S_{\gamma}.
\end{equation}

Let $X=\{0, 1\}$, and consider the corresponding standard actions on
$X^*$ and $X^\omega$. Denote by $a$ the generator of $\img{f}$
corresponding to $S_\gamma$. Then, by~\eqref{eq:admach1} and~\eqref{eq:admach2},
\[aS_0=S_1,\qquad aS_1=S_0a.\]
In other words, the action of $a$ on $X^\omega$ is given by
the recurrent formulas
\[a(0v)=1v,\qquad a(1v)=0a(v).\]
We see that $a$ acts as the \emph{binary adding machine}, see~\cite[Section~1.7.1]{nek:book}.

\begin{figure}
\centering
\includegraphics{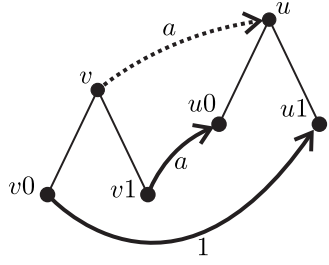}
\caption{The adding machine}
\label{fig:splittinga}
\end{figure}

The group $\V_f$ is generated by the Higman-Thompson group $\V_2$
(also coinciding with the Thompson group $V$, see~\cite{intro_tomp}) and an element
$a$ satisfying the splitting rule $a=S_1S_0^*+S_0aS_1^*$, i.e.,
\[\left(\begin{array}{c}v\\ a\\ u\end{array}\right)=
\left(\begin{array}{cc}v0 & v1\\ 1 & a\\ u1 & u0\end{array}\right).\]
See Figure~\ref{fig:splittinga}.
\end{example}

\begin{example}
\label{ex:basilica}
Consider the complex polynomial $z^2-1$ as a partial self-covering
$f:\M_1\arr\M$, where $\M=\C\setminus\{0, -1\}$, and
$\M_1=\C\setminus\{0, \pm 1\}$. Alternatively, we can consider it as a
self-covering of its Julia set, see Figure~\ref{fig:basilica}.

\begin{figure}
\centering
\includegraphics{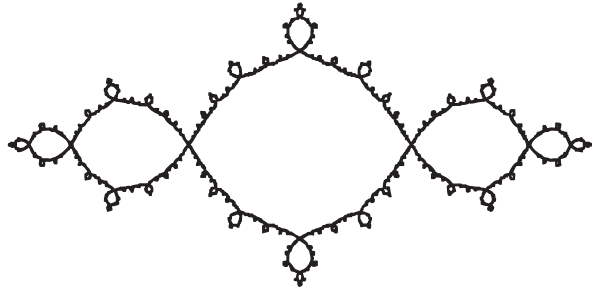}
\caption{Julia set of $z^2-1$}
\label{fig:basilica}
\end{figure}

Then $\V_{z^2-1}$ is generated by the Thompson group $\V_2$ and two
elements $a, b$ satisfying
\[a=S_1S_0^*+S_0bS_1^*,\qquad b=S_0S_0^*+S_1aS_1^*.\]

For a detailed proof of the recurrent definitions of the elements $a$
and $b$, see~\cite[Subsection~5.2.2.]{nek:book}.
\end{example}

\subsection{Self-similar groups}

The description of $\V_f$ in terms of the iterated monodromy
group given in Proposition~\ref{pr:symbolic}
can be generalized in the following way.

\begin{defi}
\label{def:ssimilar}
Let $G$ be a group acting faithfully by automorphisms of the tree
$X^*$. We say that $G$ is a \emph{self-similar group} if for all $g\in G$ and
$x\in X$ there exist $h\in G$ and $y\in X$ such that
\[g(xw)=yh(w)\]
for all $w\in X^*$.
\end{defi}

We will usually denote self-similar groups as pairs $(G, X)$.
Note that the equation in Definition~\ref{def:ssimilar} is equivalent
to the equality
\[g\cdot S_x=S_y\cdot h,\]
of compositions of self-maps of $X^\omega$,
where $S_x$ is, as before, the transformation $S_x(w)=xw$ of
$X^\omega$.

\begin{defi}
\label{def:VG}
Let $G$ be a self-similar group acting on $X^*$. The group $\V_G$ is
the set of all homeomorphisms $g$ of $X^\omega$ for which there exist
complete antichains $A_1, A_2\subset X^*$, a bijection $\alpha:A_1\arr
A_2$, and elements $g_v\in G$, for $v\in A_1$ such that
\[g=\sum_{v\in A_1}S_{\alpha(v)}g_vS_v^*,\]
i.e.,
\[g(vw)=\alpha(v)g_v(w)\]
for all $v\in A_1$ and $w\in X^\omega$.
\end{defi}

If $G$ is a self-similar group acting on $X^*$, then for every $v\in
X^*$ there exists a unique element of $G$, denoted $g|_v$, such that
\[g(vw)=g(v)g|_v(w)\]
for all $w\in X^\omega$. We call $g|_v$ the \emph{section} of $g$ in
$v$. 

The elements of $\V_G$ are represented by tables of the
form~\eqref{eq:tablesG} with the splitting
rule~\eqref{eq:tablesspittingG}. The following proposition follows now
directly from the described constructions.

\begin{proposition}
Consider $\img{f}$ as a self-similar group with respect to a standard
action on $X^*$. Then $\V_f$ (with the corresponding standard action
on $X^\omega$) is equal to $\V_{\img{f}}$.
\end{proposition}

Let us describe more examples of groups $\V_f$ and $\V_G$.

\begin{example}
\label{ex:GrigorchukG}
The Grigorchuk group $G$ is generated by the transformations
\[a=S_1S_0^*+S_0S_1^*,\quad b=S_0aS_0^*+S_1cS_1^*,
c=S_0aS_0^*+S_1dS_1^*, d=S_0S_0^*+S_1bS_1^*,\]
see~\cite{grigorchuk:80_en}.

The corresponding group $\V_G$ was defined and studied 
by C.~R\"over in~\cite{roever,roever:comm}. This was the first example
of a group $\V_G$.
C.~R\"over proved that $\V_G$ is isomorphic to the
abstract commensurator of $G$, that it is finitely presented, and
simple. We will study the last two properties of the groups $\V_G$,
generalizing the results of C.~R\"over for a wide class of self-similar groups.
\end{example}

\begin{example}
\label{ex:kneadingv}
Let $f(z)=z^2+c$ be a complex quadratic polynomial such that
$f^n(0)=0$ for some $n$ (we assume that $n$ is the smallest number
with this property). Then $f$ is a self-covering of its Julia set,
which is path connected. The iterated monodromy groups $\img{f}$
associated with such polynomials were described
in~\cite{bartnek:mand}.
There exists a
sequence $v=x_1x_2\ldots x_{n-1}\in\{0, 1\}^{n-1}$ such that $\img{f}$ is isomorphic to
the group $\mathfrak{K}_v$ generated by $n$ elements $a_0, a_1,
\ldots, a_{n-1}$ given by the recurrent relations
\[a_0=S_1S_0^*+S_0a_{n-1}S_1^*,\]
and
\[a_i=\left\{\begin{array}{cc} S_0a_{i-1}S_0^*+S_1S_1^* & \text{if $x_i=0$}\\
S_0S_0^*+S_1a_{i-1}S_1^* & \text{if $x_i=1$},\end{array}\right.\]
for $i=1, 2, \ldots, n-1$. For example, $\img{z^2-1}=\mathfrak{K}_0$.
\end{example}

\subsection{Wreath recursions}

Let $(G, X)$ be a self-similar group. Every element $g\in G$ defines a
permutation $\sigma_g$ of $X=X^1\subset X^*$, and an element of $G^X$
equal to the function $f_g:x\mapsto g|_x$. It is easy to check that 
the map $\psi:G\arr \symm{X}\ltimes G^X$ mapping $g$ to
$(\sigma_g, f_g)$ is a homomorphism of groups, which we call the
\emph{wreath recursion} associated with the self-similar group.

Let $X=\{1, 2, \ldots, d\}$. We will write elements of
$\symm{d}\ltimes G^d=\symm{X}\ltimes G^X$ as products
$\sigma(g_1, g_2, \ldots, g_d)$, where $\sigma\in\symm{d}$ and $(g_1,
g_2, \ldots, g_d)\in G^d$.
Multiplication rule for elements of the wreath
product $G\wr\symm{d}=\symm{d}\ltimes G^d$ is given by the formula
\begin{equation}
\label{eq:multiplicationwreath}
\sigma(g_1, g_2, \ldots, g_d)\pi(h_1, h_2, \ldots,
h_d)=\sigma\pi(g_{\pi(1)}h_1, g_{\pi(2)}h_2, \ldots, g_{\pi(d)}h_d).
\end{equation}

The wreath recursion completely describes the self-similar group $G$
by giving recurrent formulas for the action of its elements on $X^*$.

\begin{example}
The adding machine, see Example~\ref{ex:addingmachine},
is given by the recursion $a=\sigma(1, a)$, where
$\sigma$ is the transposition $(0, 1)$. The generators of
$\img{z^2-1}$ are given by
\[a=\sigma(1, b),\qquad b=(1, a),\]
see Example~\ref{ex:basilica}.
\end{example}

Any homomorphism
$\psi:G\arr\symm{d}\ltimes G^d$ defines an action of $G$ on $X^*$ (for
$X=\{1, 2, \ldots, d\}$) by
the recurrent rule:
\[g(xw)=\sigma(x)g_x(w),\]
where $\sigma$ and $g_x$ are defined by the condition
$\psi(g)=\sigma(g_1, g_2, \ldots, g_d)$. This action is not
faithful in general. The quotient of $G$ by the kernel of its action
on $X^*$ is called the \emph{faithful quotient} of $G$, and it is a
self-similar group in the sense of Definition~\ref{def:ssimilar}.

Let $\psi:G\arr\symm{d}\ltimes G^d$ be an arbitrary homomorphism. Then
we can define the group $\V_\psi$ associated with it in the same way
as the groups $\V_G$ were defined for self-similar groups. Namely,
elements of $\V_\psi$ are defined by tables of the form
\[\left(\begin{array}{cccc} v_1 & v_2 & \ldots & v_n\\ g_1 & g_2 &
    \ldots & g_n\\ u_1 & u_2 & \ldots & u_n\end{array}\right),\]
where $g_i\in G$, and $\{v_1, v_2, \ldots, v_n\}$ and $\{u_1, u_2,
\ldots, u_n\}$ are complete antichains of $X^*$. Two tables define the
same element if they can be made equal (up to permutations of the
columns) by iterated replacement of a
column $\left(\begin{array}{c}v\\ g\\ u\end{array}\right)$
by the columns
\[\left(\begin{array}{cccc} v1 & v2 & \ldots & vd\\ g_1 & g_2 & \ldots &
    g_d\\
u\sigma(1) & u\sigma(2) & \ldots & u\sigma(d)\end{array}\right),\]
where $\psi(g)=\sigma(g_1, g_2, \ldots, g_d)$.
Multiplication of the tables is defined by the rule
\[\left(\begin{array}{ccc} w_1 & \ldots & w_n\\ g_1 &
    \ldots & g_n\\ u_1 & \ldots & u_n\end{array}\right)
\left(\begin{array}{cccc} v_1 & \ldots & v_n\\ h_1 & \ldots & h_n\\ w_1
    & \ldots & w_n\end{array}\right)=
\left(\begin{array}{ccc} v_1 & \ldots & v_n\\ g_1h_1 & \ldots & g_nh_n\\
    u_1 & \ldots & u_n\end{array}\right).\] 

\subsection{Bisets}
A formalism equivalent to wreath recursions is provided by the notion of a
\emph{covering biset}.
If $(G, X)$ is a self-similar group, then 
the set of transformations $S_xg:w\mapsto xg(w)$ of $X^\omega$ is
invariant under the left and right multiplications by elements of $G$:
\[S_xg\cdot h=S_x(gh),\qquad h\cdot S_xg=S_{h(x)}(h|_xg).\]
We get therefore commuting left and right actions of $G$ on the set
$\Phi=\{S_xg\;:\;x\in X, g\in G\}$. 
We will write elements $S_x\cdot g$ of $\Phi$ just as $x\cdot g$.

We adopt the following definition.

\begin{defi}
Let $G$ be a group. A \emph{$G$-biset} is a set $\Phi$ together with
commuting left and right $G$-actions. It is called a \emph{covering
biset} if the right action is free (i.e., if $x\cdot g=x$ for
$x\in\Phi$ and $g\in G$ implies $g=1$) and has a finite number of
orbits.
\end{defi}

Let $\Phi_1, \Phi_2$ be $G$-bisets. Then their tensor product
$\Phi_1\otimes\Phi_2$ is defined as the quotient of the set
$\Phi_1\times\Phi_2$ by the identifications
\[(x\cdot g)\otimes y=x\otimes (g\cdot y),\qquad g\in G.\]

Let $\Phi$ be the biset $\{x\cdot g\;:\;x\in X, g\in
G\}$ associated with a self-similar group. Then every element
of $\Phi^{\otimes n}$ can be uniquely written in the form $x_1\otimes
x_2\otimes\cdots\otimes x_n\cdot g$, where $x_i=x_i\cdot 1$ are
elements of $X$. It follows that $n$th tensor power
$\Phi^{\otimes n}$ is naturally identified with the set of pairs
$v\cdot g$, for $v\in X^n$ and $g\in G$, with the actions
\[h\cdot (v\cdot g)=h(v)\cdot (h|_vg),\qquad (v\cdot g)\cdot h=v\cdot (gh).\]

Let $\Phi$ be an arbitrary covering $G$-biset. Choose a transversal
$X\subset\Phi$ of the orbits of the right action. Then every element
of $\Phi$ is uniquely written in the form $x\cdot g$ for $x\in X$ and
$g\in G$. For every $g\in G$ and $x\in X$ there exist $h\in G$ and
$y\in X$ such that \[g\cdot x=y\cdot h,\]
and the elements $y, h$ are uniquely determined by $g$ and $x$. We get
hence a homomorphism $\psi:G\arr \symm{X}\ltimes G^X$, called the
\emph{wreath recursion} associated with $\Phi$ and $X$. Namely,
$\psi(g)=\sigma\cdot f$, where $\sigma\in\symm{X}$ and $f\in G^X$ satisfy
\[g\cdot x=\sigma(x)\cdot f(x)\]
for all $x\in X$ (where
$G^X$ is seen as the set of functions $X\arr G$). If we change the
orbit transversal $X$ to an orbit transversal $Y$, then the
homomorphism $\psi:G\arr\symm{|X|}\ltimes G^{|X|}$ is composed with an
inner automorphism of the wreath product (after we identify $X$
with $Y$ by a bijection).
Note that the biset $\Phi$ is uniquely determined, up to an
isomorphism of biset, by the homomorphism $\psi$.

\begin{defi}
\label{def:ssequivalent}
Two self-similar actions $(G, X_1)$ and $(G, X_2)$ of a group $G$ are
called \emph{equivalent} if their associated bisets $\Phi_i=X_i\cdot
G$ are isomorphic, i.e., if there exists a bijection
$F:\Phi_1\arr\Phi_2$ such that $F(g_1\cdot a\cdot g_2)=g_1\cdot
F(a)\cdot g_2$ for all $g_1, g_2\in G$ and $a\in\Phi_1$. Two
self-similar actions of groups $G_1, G_2$ are equivalent if they
become equivalent after identification of the groups $G_1, G_2$ by an
isomorphism $G_1\arr G_2$.
\end{defi}

The wreath recursion can be defined invariantly, without a choice
of the orbit transversal. Namely, let $\mathop{\mathrm{Aut}}(\Phi_G)$
be the automorphism group of the right $G$-set $\Phi$, i.e., the set
of all bijections $\alpha:\Phi\arr\Phi$ such that $\alpha(x\cdot
g)=\alpha(x)\cdot g$. Then $\mathop{\mathrm{Aut}}(\Phi_G)$ is
isomorphic to the wreath product $\symm{d}\ltimes G^d$, where $d$ is the
number of the orbits of the right action on $\Phi$, since the right
$G$-set $\Phi$ is free and has $d$ orbits, i.e., is isomorphic to the
disjoint union of $d$ copies of $G$. For every element
$g\in G$ the map $\psi(g):x\mapsto g\cdot x$ is an automorphism of the right
$G$-set $\Phi$. Then $\psi:G\arr\mathop{\mathrm{Aut}}(\Phi_G)$ is the
wreath recursion. For more on wreath recursions and bisets, 
see~\cite{nek:book,nek:filling,nek:models}.

\begin{example} Let $f:\M_1\arr\M$ be a self-covering map, and let
  $\iota:\M_1\arr\M$ be a continuous map (for example, $f$ is a
  partial self-covering, and $\iota$ is the identical embedding).

Suppose that $\M$ is path-connected. Choose a basepoint $t\in\M$, and 
consider the set $\Phi$ of pairs $(z, \ell)$, where $z\in f^{-1}(t)$,
and $\ell$ is a homotopy class of a path in $\M$ from $t$ to
$\iota(z)$. Then $\pi_1(\M, t)$ acts on $\M$ be appending loops to the
beginning of the path $\ell$:
\[(z, \ell)\cdot\gamma=(z, \ell\gamma).\]
It also acts by appending images of lifts of $\gamma$ to the end of
the path $\ell$:
\[\gamma(z, \ell)=(z', \iota(\gamma_z)\ell),\]
where $\gamma_z$ is the lift of $\gamma$ by $f$ starting at $z$, and
$z'$ is the end of $\gamma_z$. Here, as before, we multiply paths as
functions (second path in a product is passed first).

Then $\Phi$ is a covering $\pi_1(\M, t)$-biset. The associated wreath
recursion coincides with the wreath recursion associated with the
standard action of $\img{f}$.
\end{example}

Let us show a more canonical definition of the groups $\V_\psi$ in terms
of bisets.
Let $\Phi$ be a covering $G$-biset. Consider the biset $\Phi^*$ equal
to the disjoint union of the bisets $\Phi^{\otimes n}$ for all
integers $n\ge 0$. Here $\Phi^{\otimes 0}$ is the group $G$ with the
natural $G$-biset structure. The set $\Phi^*$ is a semigroup with respect to the
tensor product operation.

Let us order the semigroup $\Phi^*$ with respect to the left divisibility,
i.e., $v\preceq u$ if and only if there exists $w$ such that
$u=v\otimes w$. It is easy to check that $\Phi^*$ is
left-cancellative, i.e., that $v\otimes w_1=v\otimes w_2$ implies that
$w_1=w_2$.

The quotient of $\Phi^*$ by the right $G$-action is a rooted
$d$-regular tree, and the image of $\preceq$ under the quotient map is the natural order
on the rooted tree $\Phi^*/G$. If $X$ is a right orbit transversal of
$\Phi$, then $X^{\otimes n}=X^n$ is a right orbit transversal of
$\Phi^{\otimes n}$, and the identical embedding of $X^*=\bigcup_{n\ge
  0}X^{\otimes n}$ into $\Phi^*$
induces an isomorphism of the rooted tree $X^*$ with
$\Phi^*/G$.

The left action of $G$ on $\Phi^*$ permutes the orbits of the right
action and preserves the relation $\preceq$, hence $G$ acts on the tree $\Phi^*/G$ by
automorphisms. The corresponding action on $X^*$ is the self-similar
action defined by the wreath recursion associated with $X$.

Let $A_1, A_2\subset\Phi^*$ be finite maximal
antichains with respect to the divisibility order $\preceq$. Note that
a subset $A\subset\Phi^*$ is a finite maximal antichain if and only if
its image in $\Phi^*/G$ is a maximal antichain. Choose a
bijection $\alpha:A_1\arr A_2$. 

If $w\in\Phi^*$ is such that $v\preceq
w$ for some $v\in A_1$, then there exists a unique $u\in\Phi^*$ such
that $w=v\otimes u$. Consider then the transformation $h_{A_1, \alpha,
  A_2}:w\mapsto
\alpha(v)\otimes u$. The map $h_{A_1, \alpha, A_2}$ is defined for all elements of
$\Phi^*$ bigger than some element of $A_1$, hence for all elements of
$\Phi^{\otimes n}$, where $n$ is big enough. We will identify to
transformation $h_{A_1, \alpha, A_2}$ and $h_{A_1', \alpha', A_2'}$ if
their actions on the sets $\Phi^{\otimes n}$ agree for all $n$ big
enough. It is not hard to prove that the set of equivalence classes of
such maps is a group, which we will denote $\V_\Phi$. It is also
straightforward to show that $\V_\Phi$ coincides with $\V_\psi$, where
$\psi$ is the wreath recursion associated with $\Phi$, and that 
if $\Phi$ is the usual biset associated
with a self-similar group $G$, then $\V_\Phi$ coincides with $\V_G$.

\subsection{Epimorphism onto the faithful quotient}

Consider a covering $G$-biset $\Phi$ and the corresponding group $\V_\Phi$.
The faithful quotient $\overline G$ is a self-similar group acting on $X^*$. We
have, therefore two groups: $\V_\Phi$ and $\V_{\overline G}$, which
are non-isomorphic in general (the group $\V_{\overline G}$ is a
homomorphic image of $\V_\Phi$).

\begin{proposition}
\label{pr:kerneln}
Let $\Phi$ be a covering biset. Denote by $K_n$ the subgroup of
elements of $G$ acting trivially from the left on $\Phi^{\otimes n}$,
i.e., the kernel of the wreath recursion associated with the biset
$\Phi^{\otimes n}$. Then $K_n\supseteq K_{n-1}$. If $\bigcup_{n\ge
  0}K_n$ is equal to the kernel of the epimorphism
$G\arr\overline{G}$, then the natural epimorphism
$\V_\Phi\arr\V_{\overline G}$ is an isomorphism.
\end{proposition}

\begin{proof}
Suppose that an element $g$ of the kernel of $\V_\Phi\arr\V_{\overline
  G}$ is defined by a table $\left(\begin{array}{cccc}v_1 & v_2 &
    \ldots & v_n\\ g_1 & g_2 & \ldots & g_n\\ u_1 & u_2 & \ldots &
    u_n\end{array}\right)$. Then the table $\left(\begin{array}{cccc}v_1 & v_2 &
    \ldots & v_n\\ \overline{g_1} & \overline{g_2} & \ldots & \overline{g_n}\\ u_1 & u_2 & \ldots &
    u_n\end{array}\right)$ represents the trivial element of
$\V_{\overline G}$, where $g\mapsto\overline g$ is the epimorphism
$G\arr\overline G$. But this means that $v_i=u_i$ and
$\overline{g_i}=1$ for all $i$. Consequently, there exists $k$ such
that $g_i\in K_k$ for all $i$. It follows that after applying elementary
splittings $k$ times to all column of the table defining $g$, we will
get a table defining the trivial element of $\V_\Phi$, which means
that $g=1$.
\end{proof}

\section{Simplicity of the commutator subgroup}

\subsection{Some general facts}
\label{ss:somegeneralfacts}

Let $G$ be a group acting faithfully by homeomorphisms on an infinite
Hausdorff space $\mathcal{X}$. For an open subset $U\subset\mathcal{X}$, denote
by $G_{(U)}$ the subgroup of elements of $G$ acting trivially on
$\mathcal{X}\setminus U$. 
Denote by $R_U$ the normal closure in $G$ of the derived subgroup
$G_{(U)}{}'=[G_{(U)}, G_{(U)}]$.

The following simple lemma has appeared in many papers in different
forms, see, for example \cite[Lemma~5.3]{handbook:branch}, 
\cite[Theorem~4.9]{matui:fullI}.

\begin{lemma}
\label{lem:basic}
Let $N$ be a non-trivial normal subgroup of $G$. Then there exists a
non-empty open subset $U\subset\mathcal{X}$ such that $R_U\le N$.
\end{lemma}

\begin{proof}
It is sufficient to prove that there exists an open subset $U$ such
that $G_{(U)}{}'\le N$.

Let $g\in N\setminus\{1\}$, and let $x\in\mathcal{X}$ be such that
$g(x)\ne x$. Then there exists an open subset $U$ such that $x\in U$
and $U\cap gU=\emptyset$. For example, find disjoint neighborhoods $U_x$
and $U_{g(x)}$ of $x$ and $g(x)$, and set $U=U_x\cap
g^{-1}(U_{g(x)})$. 

Let $h_1, h_2\in G_{(U)}$. Then $gh_1^{-1}g^{-1}$ acts trivially
outside $gU$. Consequently, $[g^{-1}, h_1]=gh_1^{-1}g^{-1}h_1$ acts as $h_1$
on $U$, as $gh_1^{-1}g^{-1}$ on $gU$, and trivially outside 
$U\cup gU$.  It follows that $[[g^{-1}, h_1], h_2]$ acts as $[h_1, h_2]$
on $U$ and trivially outside $U$, i.e., $[[g^{-1}, h_1], h_2]=[h_1,
h_2]$. On the other hand, $[[g^{-1}, h_1], h_2]\in N$, since $N$ is normal
and $g\in N$. It follows that $[h_1, h_2]\in N$ for all $h_1, h_2\in
G_{(U)}$, i.e., that $G_{(U)}{}'\le N$.
\end{proof}

\begin{lemma}
\label{lem:Ru}
Let $U$ be an open subset of $\mathcal{X}$ such that its $G$-orbit is
a basis of topology of $\mathcal{X}$. Then every non-trivial normal subgroup of
$G$ contains $R_U$.
\end{lemma}

\begin{proof}
For any $g\in G$ and $U\subset\mathcal{X}$, we have
$gG_{(U)}g^{-1}=G_{(gU)}$. Consequently,
$gG_{(U)}{}'g^{-1}=G_{(gU)}{}'$, and $R_U$ is
equal to the group generated by $\bigcup_{g\in G}G_{(gU)}{}'$. In
particular, $R_{gU}=R_U$ for all $g\in G$.

Since $\{gU\;:\;g\in G\}$ is a basis of topology, for every open
subset $W\subset\X$ there exists $g\in G$ such that $gU\subseteq W$,
hence $R_U=R_{gU}\le R_W$. This implies, by Lemma~\ref{lem:basic},
that $R_U$ is contained in every non-trivial normal subgroup of $G$.
\end{proof}

\begin{proposition}
\label{pr:RU}
Suppose that $U$ is an open subset of $\X$ such that its $R_U$-orbit
is equal to its $G$-orbit and
is a basis of topology of $\X$. Then $R_U$ is simple and is contained
in every non-trivial normal subgroup of $G$.
\end{proposition}

\begin{proof}
The group $G_{(U)}{}'$ is non-trivial, since
otherwise its normal closure $R_U$ is trivial, which contradicts the
fact that the $R_U$-orbit of $U$ is a basis of topology. It follows that
there exists $g\in G_{(U)}{}'$ moving a point $x\in\X$. We have then $x\in U$
and $g(x)\in U$. There exists a neighborhood $W$ of $x$ such that
$W\cap gW=\emptyset$. Then there exists an element $h\in R_U$ such
that $hU\subset W$, and there exists a non-trivial element $g'\in
G_{(hU)}{}'\le G_{(U)}{}'$. We have $[g, g']\ne 1$, since $g'$ and $gg'g^{-1}$
have disjoint supports. This shows that $G_{(U)}{}'$ is non-abelian,
i.e., that $G_{(U)}{}''$ is non-trivial.

The subgroup $R_U$ is contained in every
non-trivial normal subgroup of $G$, by Lemma~\ref{lem:Ru}.
We also have that the normal closure in $R_U$ of the group
$(R_U)_{(U)}{}'$ is contained in every normal subgroup of $R_U$.

Note that $G_{(U)}{}'$ is contained in $G_{(U)}\cap R_U=(R_U)_{(U)}$,
hence $G_{(U)}{}''\le(R_U)_{(U)}{}'$. Conjugating by an element $g\in
G$ (and using that
$R_U$ is normal in $G$), we get that $G_{(gU)}{}''\le (R_U)_{(gU)}{}'$.

It follows that for any non-trivial subgroup $N\unlhd R_U$ we have
 \[N\supseteq\bigcup_{g\in
  R_U}(R_U)_{(gU)}{}'\supseteq\bigcup_{g\in R_U}G_{(gU)}{}''=\bigcup_{g\in
  G}G_{(gU)}{}''.\]
Consequently, $N$ contains the group generated by the set $\bigcup_{g\in
  G}G_{(gU)}{}''$, which is normal in $G$ and non-trivial, hence
contains $R_U$. We have proved that every non-trivial normal subgroup of $R_U$
contains $R_U$, i.e., that $R_U$ is simple.
\end{proof}

\subsection{Simplicity of $\V_G'$}

Let $(G, X)$ be a self-similar group, and let $\V_G$ be the corresponding
group of homeomorphisms of the Cantor set $X^\omega$.

Fix a linear ordering ``$<$'' of the elements of
$X$. Extend it to the lexicographic ordering on $X^*$. Namely, if
$x_1x_2\ldots x_n$ and $y_1y_2\ldots y_m$ are incomparable with
respect to the order $\preceq$, then
$x_1x_2\ldots x_n<y_1y_2\ldots y_m$ if and only if $x_i<y_i$, where
$i$ is the smallest index such that $x_i\ne y_i$. If a word $v$ is a
beginning of a word $w$, then $v\le w$.

Suppose that $\left(\begin{array}{cccc}v_1 & v_2 & \ldots & v_n\\ u_1 &
u_2 & \ldots & u_n\end{array}\right)$ is a table defining an
element $g\in\V_X$, i.e., that $g=\sum_{i=1}^nS_{u_i}S_{v_i}^*$.
Assume that its first row is ordered in the increasing
lexicographic order. If the permutation putting the second row into
the increasing lexicographic order is even, then we say that the table
is even.

Note that if $d=|X|$ is even, then every table has an even splitting. On the
other hand, if $d$ is odd, then the set of elements defined by even
tables is a subgroup of index 2 in $\V_d$. The following is proved
in~\cite{hgthomp}.

\begin{theorem}
\label{th:vdprime}
If $d$ is even, then $\V_d$ is simple. If $d$ is odd, then the
commutator subgroup $\V_d'$ is the group of elements defined by even
tables, and is a simple subgroup of index 2 in $\V_d$.
\end{theorem}

The following theorem is a proved in~\cite[Theorem~9.11]{nek:bim}.

\begin{theorem}
\label{th:vgabelian}
All proper quotients of $\V_G$ are abelian.
\end{theorem}

Let us describe the $\V_d$-orbits of clopen subsets of $X^\omega$.

\begin{proposition}
\label{pr:VXorbits}
Let $d=|X|$, and let $U$ be a clopen subset of $X^\omega$. Let us
decompose $U$ into a disjoint union $\bigsqcup_{i=1}^n v_iX^\omega$
for $v_i\in X^*$. Then the residue of $n$ modulo $d-1$ does not depend
on the decomposition. Let us denote it $m(U)\in\Z/(d-1)\Z$. Let $U_1,
U_2\subset X^\omega$ be non-empty clopen proper subsets. Then the
following conditions are equivalent.
\begin{enumerate}
\item $U_1$ and $U_2$ belong to one $\V_X$-orbit.
\item $U_1$ and $U_2$ belong to one $\V_X'$-orbit.
\item $m(U_1)=m(U_2)$.
\end{enumerate}
\end{proposition}

\begin{proof}
Let $U=\bigsqcup_{i=1}^n v_iX^\omega$ be a decomposition of $U$. We can
\emph{split} it by replacing one set $v_iX^\omega$ by the collection
of $d$ sets $v_ixX^\omega$, for $x\in X$. This way we increase the
number of sets in the decomposition by $d-1$. It is easy to see that
for any two decompositions of $U$ into cylindrical sets 
there exist sequences of successive splittings of each of
the decompositions leading to the same decomposition. This implies
that $m(U)$ is well defined.

It is obvious that $m(U)$ is preserved under the action of
$\V_X$. Suppose that $m(U_1)=m(U_2)$ for non-empty proper clopen
subsets of $X^\omega$. Then there exist decompositions
$U_1=\bigsqcup_{i=1}^nv_iX^\omega$ and $U_2=\bigsqcup_{i=1}^n
u_iX^\omega$ of the sets $U_i$ into equal number of cylindrical
subsets. The sets $\{v_i\}_{i=1}^n$ and $\{u_i\}_{i=1}^n$ are
incomplete antichains in $X^*$, hence (see Lemma~\ref{lem:incomplete})
there exists $g\in \V_X$ such that
$g(v_iX^\omega)=u_iX^\omega$ for all $i$. If $g\notin\V_X'$, then
we can compose it with an element $h\in\V_X\setminus\V_X'$ acting
trivially on $\bigsqcup_{i=1}^nv_iX^\omega$, and get an element
$gh\in\V_X'$ such that $gh(v_iX^\omega)=u_iX^\omega$ for all $i$. It
follows that $U_1$ and $U_2$ belong to one $\V_X'$-orbit and to one
$\V_X$-orbit.
\end{proof}

It easily follows from the definitions that $\V_X$-orbits of clopen
subsets of $X^\omega$ coincide with their $\V_G$-orbits for any
self-similar group $(G, X)$.

\begin{theorem}
\label{th:commutatorsimple}
The commutator subgroup of $\V_G$ is simple.
\end{theorem}

\begin{proof}
Let $U=xX^\omega$ for some $x\in X$. Then the $\V_G$-orbit of $U$
coincides with its $\V_X$-orbit and its $\V_X'$-orbit, and is a basis
of the topology on $X^\omega$, see Proposition~\ref{pr:VXorbits}. The
group $\V_X'$ is contained in $R_U$, since the normal closure of
$(\V_X)_{(U)}{}'$ in $\V_X$ is equal to $\V_X'$, by
Theorem~\ref{th:vdprime}. It follows that the $R_U$-orbit of $U$ is
equal to its $G$-orbit, and is a basis of the topology.

Consequently, by Proposition~\ref{pr:RU}, $R_U$ is simple and is
contained in every non-trivial normal subgroup of $\V_G$. On the other hand, by
Theorem~\ref{th:vgabelian}, every non-trivial normal subgroup of
$\V_G$ contains the commutator subgroup $\V_G{}'$. It follows that
$\V_G'=R_U$, and $\V_G'$ is simple.
\end{proof}

Abelianization of $\V_G$ is described in~\cite[Theorem~9.14]{nek:bim}.

\begin{theorem}
\label{th:abelquotient}
Let $(G, X)$ be a self-similar group. Let $\pi:G\arr G/G'$ be the
abelianization epimorphism.

Suppose that $d$ is even. Then $\V_G/\V_G'$ is isomorphic to the
quotient of $G/G'$ by the relations
$\pi(g)=\sum_{x\in X}\pi(g|_x)$ for all $g\in G$.

Suppose that $d$ is odd. Then $\V_G/\V_G'$ is isomorphic to
$\Z/2\Z\oplus G/G'$ modulo the relations
$\pi(g)=\mathop{\mathrm{sign}}(g)\oplus\sum_{x\in X}\pi(g|_x)$ for
all $g\in G$; where $\mathop{\mathrm{sign}}(g)=0$ if $g$ defines an even permutation on
the first level $X$ of $X^*$, and $\mathop{\mathrm{sign}}(g)=1$ otherwise.
\end{theorem}

Proof of the following lemma follows directly from the multiplication
rule~\eqref{eq:multiplicationwreath}.

\begin{lemma} 
\label{lem:transfer}
Let $(G, X)$ be a self-similar group, and let
$\pi:G\arr G/G'$ be the canonical homomorphism.
Then $\sigma:\pi(g)\mapsto\sum_{x\in X}\pi(g|_x)$ is a
well defined endomorphism of $G/G'$.
\end{lemma}

For odd $d$ denote by $\sigma_1$ the endomorphism $(t, g)\mapsto
(t+\mathop{\mathrm{sign}}(g), \sigma(g))$ of $\Z/2\Z\oplus G/G'$, where
$\mathop{\mathrm{sign}}:G/G'\arr\Z/2\Z$ gives the parity of the action
of any preimage of $g$ in $G$ on the first level $X$ of the tree $X^*$.

Then Theorem~\ref{th:abelquotient} can be reformulated as follows. (We
denote here the identical automorphism of a group by $1$.)

\begin{theorem}
\label{th:abelquotient2}
If $d$ is even, then $\V_G/\V_G'$ is isomorphic to the
quotient of $G/G'$ by the range of the endomorphism $1-\sigma$.

If $d$ is odd, then $\V_G/\V_G'$ is isomorphic to the
quotient of $\Z/2\Z\oplus G/G'$ by the range of the endomorphism $1-\sigma_1$.
\end{theorem}

\begin{example}
Consider the double self-covering $f(x)=2x$ of the circle
$\R/\Z$. Then $\img{f}$ is generated by the adding machine
$a=S_1S_0^*+S_0aS_1^*$.
It follows that $\V_f/\V_f'$ is the quotient of $\Z=\img{f}/\img{f}'$ by the range of
$1-\sigma$, where $\sigma(a)=0+a=a$ is the identity
isomorphism. Hence, $1-\sigma=0$, and $\V_f/\V_f'\cong\Z$.
\end{example}

\begin{example}
Let $G=\mathfrak{K}_{x_1x_2\ldots x_{n-1}}$ be the iterated monodromy
group of a quadratic polynomial, as in Example~\ref{ex:kneadingv}. It
is proved in~\cite[Proposition~3.3]{bartnek:mand}
that $G/G'$ is the free abelian group $\Z^n$ freely generated by
the images of the generators $a_i$ of $G$. The recursive
relations defining $\mathfrak{K}_{x_1x_2\ldots x_{n-1}}$ show that
$\V_G/\V_G'$ is $\Z^n=\langle\pi(a_0), \ldots, \pi(a_{n-1})\rangle$
modulo the relations $\pi(a_0)=\pi(a_1)$, and $\pi(a_i)=\pi(a_{i-1})$
for all $i=1, 2, \ldots, n-1$. Consequently, $\V_G/\V_G'$ is
isomorphic to $\Z$.
\end{example}

\begin{example}
Let $G$ be the Grigorchuk group, see Example~\ref{ex:GrigorchukG}.
It is generated by 
\[a=S_1S_0^*+S_0S_1^*, b=S_0aS_0^*+S_1cS_1^*,
c=S_0aS_0^*+S_1dS_1^*, d=S_0S_0^*+S_1bS_1^*.\]

It follows that $\V_G/\V_G'$ is the group $G/G'$ modulo the relations
\[\pi(a)=0, \pi(b)=\pi(a)+\pi(c), \pi(c)=\pi(a)+\pi(d),
\pi(d)=\pi(c),\]
which are equivalent to $\pi(a)=0, \pi(b)=\pi(c)=\pi(d)$. It is easy
to check that $b=cd$ and $b^2=1$ in $G$, which implies
\[\pi(c)=\pi(d)=\pi(b)=\pi(c)+\pi(d)=\pi(b)+\pi(b)=\pi(b^2)=0.\]
Consequently, $\V_G/\V_G'$ is trivial, hence $\V_G$ is
simple. Simplicity of $\V_G$ was proved in~\cite{roever}.
\end{example}

\section{Expanding maps and finite presentation}

\subsection{Contracting self-similar groups and expanding maps}

\begin{defi}
\label{def:nucleus}
Let $(G, X)$ be a self-similar group. We say that it is
\emph{contracting} if there exists a finite set $\nuke\subset G$ such
that for every $g\in G$ there exists $n$ such that $g|_v\in\nuke$ for
all $v\in X^*$, $|v|\ge n$.
\end{defi}

More generally, let $\Phi$ be a covering $G$-biset, and let $X$ be a
right orbit transversal. Define then $g|_v$ for $g\in G$ and $v\in
X^n\subset\Phi^{\otimes n}$ as the unique element of $G$ such that
$g\cdot v=u\cdot g|_v$ for some (also unique) $u\in X^n$. We say that
$\Phi$ is contracting (or \emph{hyperbolic}) if there exists a finite
set $\nuke$ such that for every $g\in G$ there exists $n$ such that
$g|_v\in\nuke$ for all $v\in X^*$ such that $|v|\ge n$. One can show
(see~\cite[Corollary~2.11.7]{nek:book}) that this property depends only on $\Phi$ (but the set
$\nuke$ will depend on $X$).

The smallest set $\nuke$ satisfying the conditions of
Definition~\ref{def:nucleus} is called the \emph{nucleus} of the
self-similar action.

\begin{defi}
\label{def:expandingmap}
Let $f:\M_1\arr\M$ be a partial self-covering such that $\M$ is
compact.
The covering is called
\emph{expanding} if there exist $L>1$, $\epsilon>0$, a positive
integer $n$, and a metric $|x-y|$ on $\M$ such
that \[|f^n(x)-f^n(y)|\ge L|x-y|\]
for all $x, y\in\M_n$ such that $|x-y|<\epsilon$.
\end{defi}

For example, if $\M$ is a connected Riemann manifold, and
$\|Df^n(\xi)\|\ge CL^n\|\xi\|$ for all $n\ge 1$ and everly tangent vector
$\xi$, then $f$ is expanding.

The following proposition is proved in~\cite[Theorem~5.5.3]{nek:book}
for the case when $\M$ is a complete length metric space with finitely
generated fundamental group. We will repeat here the proof in a more
general situation (but avoiding orbispaces, which was one of the technical
issues in~\cite{nek:book}).

\begin{proposition}
\label{pr:expcontracting}
Let $f:\M_1\arr\M$ be an expanding partial self-covering of a compact path
connected space. Then $\img{f}$ is a contracting self-similar group
(with respect to any standard action).
\end{proposition}

\begin{proof}
Let $\{\ell_i\}_{i=1, \ldots, d}$ be paths connecting the basepoint
$t\in\M$ to the preimages $z_i\in f^{-1}(t)$. We consider the standard
action of $\img{f}$ on $X^*$ defined by these connecting paths, where
$X=\{1, \ldots, d\}$.

It follows from Definition~\ref{def:expandingmap} that
there exist $\epsilon>0, L>1, C>1$ such that for any subset $A\subset\M$ of
diameter less than $\epsilon$ and every $n\ge 1$, the set $f^{-n}(A)$ is a
disjoint union of $d^n$ sets $A_i$ such that $f^n:A_i\arr A$ are
homeomorphisms, and diameters of $A_i$ are less than $CL^{-n}$.

Let $g\in\img{f}$ be defined by a loop $\gamma$. Since we can
represent $\gamma$ as a union of sub-paths of diameter less than
$\epsilon$, there exists $C_\gamma>1$ such that
diameter of any lift of $\gamma$ by $f^n$ is less than $C_\gamma L^{-n}$.

By Lemma~\ref{lem:ssimilarity}, section $g|_{i_1i_2\ldots i_n}$ is
defined by the loop $\ell_{j_1j_2\ldots j_n}^{-1}\gamma_{i_1i_2\ldots
  i_n}\ell_{i_1i_2\ldots i_n}$, where $\gamma_{i_1i_2\ldots i_n}$ is a
lift of $\gamma$ by $f^n$, and  $\ell_{i_1i_2\ldots i_n}$,
$\ell_{j_1j_2\ldots j_n}$ are paths inside the tree $\Gamma_t$ formed
by lifts of the connecting paths $\ell_i$. Since diameters of lifts of
paths by $f^n$ exponentially decrease with $n$, there exists $n_0$
(depending only on the connecting paths $\ell_i$) such that if $n_0\le
n<m$, then for all
sequences $i_1i_2\ldots i_m\in X^*$, the
path $\ell_{i_1i_2\ldots i_m}$ is a continuation of the path
$\ell_{i_1i_2\ldots i_n}$ by a path of diameter less than
$\epsilon/6$.

There exists $n_1\ge n_0$ (depending on $\gamma$) 
such that if $n\ge n_1$, then all lifts of $\gamma$
by $f^n$ have diameters less than $\epsilon/6$. Then for every $n\ge
n_1$, the section $g|_{i_1i_2\ldots i_n}$ is defined by a loop of the
form $\beta=\ell_{j_1j_2\ldots j_{n_0}}^{-1}\alpha\ell_{i_1i_2\ldots i_{n_0}}$,
where $\alpha$ is a path of diameter less than
$\epsilon/2$. If $\beta'=\ell_{j_1j_2\ldots j_{n_0}}^{-1}
\alpha'\ell_{i_1i_2\ldots i_{n_0}}$ is another path such
that $\alpha'$ has diameter less than $\epsilon/2$, then
\[\beta'\beta^{-1}=\ell_{j_1j_2\ldots
  j_{n_0}}^{-1}\alpha'\alpha^{-1}\ell_{j_1j_2\ldots j_{n_0}},\]
where $\alpha'\alpha^{-1}$ is a loop of diameter less than
$\epsilon$. By the choice of $\epsilon$, all lifts of $\alpha'\alpha^{-1}$
by iterations of $f$ are loops, which implies that $\beta$ and
$\beta'$ define equal elements of $\img{f}$. Consequently,
$g|_{i_1i_2\ldots i_n}$ is uniquely determined by the pair
$(i_1i_2\ldots i_{n_0}, j_1j_2\ldots j_{n_0})$. Since $n_0$ does not
depend on $g$, it follows that $\img{f}$ is contracting.
\end{proof}

\subsection{General definition of $\V_f$ for expanding maps}
\label{ss:generaldefinition}

The group $\V_f$ can be defined for an expanding self-covering
$f:\M\arr\M$, even if $\M$ is not path connected.

Let $f:\M\arr\M$ be an expanding self-covering of a compact metric
space $\M$, such that for every $t\in\M$ the set $\bigcup_{n\ge
  0}f^{-n}(t)$ is dense in $\M$. This condition is always satisfied
for a path-connected space $\M$.

Let $\delta>0$ be such that for any two points $t_1, t_2\in\M$ such
that $t_1\ne t_2$ and $f(t_1)=f(t_2)$ we have $|t_1-t_2|>\delta$. It
is easy to prove that such $\delta$ exists for any self-covering
$f:\M\arr\M$ of a compact metric space. It follows from the definition
of an expanding self-covering, that there exists $\epsilon>0$ such
that for any two points $z_1, z_2\in\M$ and for any $n\ge 1$ there
exists an isomorphism $S_{z_1, z_2}:T_{z_1}\arr T_{z_2}$ of the trees
of preimages such that $|S_{z_1, z_2}(v)-v|<\delta/2$ for all $v\in
T_{z_1}$. Moreover, it is easy to prove (by induction on the level number)
that the isomorphism $S_{z_1, z_2}$ is unique.

Fix a basepoint $t\in\M$, and define $\V_f$ as the group of
homeomorphisms of $\partial T_t$ piecewise equal to the isomorphisms
$S_{z_1, z_2}:\partial T_{z_1}\arr\partial T_{z_2}$ for $z_1, z_2\in
T_t$.

Note that if $\gamma:[0, 1]\arr\M$ is a path in $\M$, then there
exists $n$ such that all lifts of $\gamma$ by $f^m$ for $m\ge n$ have
diameter less than $\delta/2$. This implies that if $\M$ is path
connected, then our original definition of $\V_f$ agrees with the
given definition for expanding maps.

\begin{example}
Consider the one-sided shift $\mathsf{s}:X^\omega\arr X^\omega$ 
\[\mathsf{s}(x_1x_2\ldots)=x_2x_3\ldots.\] Consider the metric
$|w_1-w_2|=2^{-n}$, where $n$ is the smallest index for which $x_n\ne
y_n$, where $w_1=x_1x_2\ldots$ and $w_2=y_1y_2\ldots$.

Then $|\mathsf{s}(w_1)-\mathsf{s}(w_2)|=2|w_1-w_2|$ whenever
$|w_1-w_2|\le 1/2$. Consequently, $\mathsf{s}$ is expanding. 

For any
$w\in X^\omega$ the set $\mathsf{s}^{-n}(w)$ is equal to the set of
sequences of the form $vw$, where $v\in X^n$. Hence, we can identify
the $n$th level of the tree $T_w$ with the set $X^n$ by the map
$vw\mapsto v$. Note that the tree $T_w$ after this identification
becomes the left Cayley graph of the monoid $X^*$: two vertices are
connected by an edge if and only if they are of the form $v, xv$ for
$v\in X^*$, $x\in X$. In particular, the boundary $\partial T_w$ is
naturally identified with the space $X^{-\omega}$ of left-infinite
sequences $\ldots x_2x_1$.

It follows directly from the definitions that
the maps $S_{wv_1, wv_2}:T_{wv_1}\arr T_{wv_2}$ act by the rule
\[S_{wv_1, wv_2}(uv_1)=uv_2.\]
It follows that the homeomorphism $\ldots x_2x_1\mapsto x_1x_2\ldots$
of $X^{-\omega}=\partial T_w$ with $X^\omega$ conjugates the action of
$\V_{\mathsf{s}}$ with the Higman-Thompson group $\V_X$.
\end{example}

\begin{example}
If $\M$ is not connected, then a covering $f:\M\arr\M$ needs not to be
of constant degree. For example, $f:\M\arr\M$ can be a one-sided shift
of finite type. The corresponding group $\V_f$ is the topological
full groups of a shift of finite type (the dual of $f$). These groups were
studied in~\cite{matui:fullonesided}.
\end{example}

\subsection{Abelianization of $\V_f$ in expanding case}

Self-similar contracting groups acting faithfully on $X^*$ are
typically infinitely presented. 
On the other hand, for every contracting group $G$ there exists a
finitely presented group $\tilde G$ and a hyperbolic covering $\tilde G$-biset
$\Phi$ such that the faithful quotient of $\tilde G$ is $G$. More
precisely, we have the following description of $\tilde G$, given
in~\cite[Section~2.13.2.]{nek:book}.

\begin{proposition}
\label{pr:lengththree}
Let $(G, X)$ be a contracting group. Suppose that the nucleus $\nuke$ 
generates $G$.

Let $\tilde G$ be the group given by the presentation $\langle
\nuke\;|\;R\rangle$, where $R$ is the set of all relations
$g_1g_2g_3=1$ of length at most 3 that hold for elements of $\nuke$ in
$G$. Let $\Phi$ be the $\tilde G$-biset of pairs $x\cdot g$, for $x\in
X$ and $g\in\tilde G$ with the actions given by the usual rules:
\[(x\cdot g)\cdot h=x\cdot (gh),\qquad h\cdot (x\cdot g)=h(x)\cdot
(h|_xg),\]
where $g\in\tilde G$, $h\in\nuke$, $x\in X$; and $h(x)\in X$, $h|_x\in\nuke$ are
defined as in $G$. Then $\Phi$ is contracting.
\end{proposition}

Note that for any contracting group $G$, the group generated by the
nucleus $\nuke$ is self-similar contracting, and
$\V_{\langle\nuke\rangle}=\V_G$.

The following is proved in~\cite[Proposition~2.13.2]{nek:book}.

\begin{proposition}
\label{pr:kernelcontracting}
Let $\Phi$ be a contracting $G$-biset. Let $\rho:G\arr\overline G$ be
the canonical epimorphism onto the faithful quotient of $G$. If
$\rho(g)\ne 1$ for every non-trivial element of the nucleus of $G$
(defined using some right orbit transversal $X$),
then the kernel of $\rho$ is equal to the union of the kernels $K_n$ 
of the left actions of $G$ on $\Phi^{\otimes n}$.
\end{proposition}

Let $\Phi$ be a $G$-biset, and let $d$ be the number of
orbits of the right action of $G$ on $\Phi$. Choose a right orbit
transversal $X\subset\Phi$, and define, for $g\in G$ and $x\in X$, the
section $g|_x$ as the unique element of $G$ such that $g\cdot x=y\cdot
g|_x$ for $y\in X$. Let $\pi:G\arr G/G'$ be the abelianization
epimorphism.

By the same arguments as in Lemma~\ref{lem:transfer}, 
the map $\sigma:\pi(g)\mapsto\sum_{x\in X}\pi(g|_x)$ is a well defined
endomorphism of $G/G'$. It is also checked directly that it does not
depend on the choice of the right orbit transversal $X$.
If $d$ is odd, then define homomorphism
$\mathop{\mathrm{sign}}:G/G'\arr\Z/2\Z$ as in
Theorem~\ref{th:abelquotient}.

The following is a direct corollary of
Propositions~\ref{pr:kernelcontracting},~\ref{pr:kerneln}, and
Theorem~\ref{th:abelquotient2}. 

\begin{corollary}
\label{cor:abelquotcontracting}
Let $\Phi$ be a $G$-biset satisfying the conditions of
Proposition~\ref{pr:kernelcontracting}.

If $d$ is even, then $\V_\Phi/\V_\Phi'$ is isomorphic to the quotient
of $G/G'$ by the range of the homomorphism $1-\sigma$. If $d$ is odd,
then $\V_\Phi/\V_\Phi'$ is isomorphic to the quotient of $\Z/2\Z\oplus
G/G'$ by the range of the endomorphism $1-\sigma_1$, where
$\sigma_1(t, g)=(t+\mathop{\mathrm{sign}}(g), \sigma(g))$.
\end{corollary}

\begin{proposition}
\label{pr:expandingpi1}
Suppose that $f:\M_1\arr\M$ is expanding, $\M$ is path-connected and
semi-locally simply connected. Then the $\pi_1(\M)$-biset associated with $f$
is contracting.

If $g\in\pi_1(\M)$ has trivial image in $\img{f}$, then there exists
$n$ such that $g$ acts trivially from the left on $\Phi_f^{\otimes
  n}$.
\end{proposition}

The first paragraph of the proposition is proved in the same way as 
Proposition~\ref{pr:expcontracting}. The second paragraph follows
directly from exponential decreasing of diameters of lifts of paths by
iterations of $f$ and the condition that $\M$ is semi-locally simply
connected.

\begin{corollary}
Suppose that $f:\M_1\arr\M$ satisfies the conditions of
Proposition~\ref{pr:expandingpi1}, and let $\Phi$ be the
$\pi_1(\M)$-biset associated with $f$. Then $\V_\Phi$ is isomorphic to
$\V_f=\V_{\img{f}}$.
\end{corollary}

Let $f:\M_1\arr\M$ be a partial self-covering satisfying the conditions of
Proposition~\ref{pr:expandingpi1}. Let $\iota:\M_1\arr\M$ be the
identical embedding. 

The group
$\pi_1(\M)/\pi_1(\M)'$ is naturally isomorphic to the first homology
group $H_1(\M)$. The map $\sigma:H_1(\M)\arr H_1(\M)$ from
Corollary~\ref{cor:abelquotcontracting} is equal to $\iota_*\circ
f^!$, where $f^!:H_1(\M)\arr H_1(\M_1)$ is the map (called the
\emph{transfer map}) given by the condition that image of a chain $c$
is equal to its full preimage $f^{-1}(c)$.

Suppose that $c\in H_1(\M)$ is
defined by a loop $\gamma$. Then parity of the monodromy action of
$\gamma$ on fibers of $f$ is well defined and generates a homomorphism
from $H_1(\M)$ to $\Z/2\Z$. Let us denote it by $\mathop{\mathrm{sign}}:H_1(\M)\arr\Z/2\Z$.

Then the following description of $\V_f/\V_f'$ 
follows directly from Corollary~\ref{cor:abelquotcontracting}.

\begin{proposition}
\label{pr:homology}
Suppose that $f:\M_1\arr\M$ is expanding, $\M$ is path-connected and
semi-locally simply connected. 

If $\deg f$ is even, then $\V_f/\V_f'$ is isomorphic to the quotient
of $H_1(\M)$ by the range of the endomorphism $1-\iota_*\circ f^!$.

If $\deg f$ is odd, then $\V_f/\V_f'$ is isomorphic to the quotient of
$\Z/2\Z\oplus H_1(\M)$ by the range of the endomorphism $1-\sigma_1$,
where $\sigma_1(t, c)=(t+\mathop{\mathrm{sign}}(c), \iota_*\circ f^!(c))$.
\end{proposition}

\subsection{Example: Hyperbolic rational functions}

Let $f$ be a complex rational function, and let $C_f$ be the set of
critical points of $f:\widehat\C\arr\widehat\C$. The
\emph{post-critical set} of $f$ is the union
$P_f=\bigcup_{n\ge 1}f^n(C_f)$ of forward orbits of critical values.
Suppose that $P_f$ is finite (we say
then that $f$ is \emph{post-critically finite}). 

Let us additionally suppose that every cycle of $f:P_f\arr P_f$ contains a critical
point. Then $f$ is \emph{hyperbolic}, i.e., is expanding on a
neighborhood of its Julia set, see~\cite[Section~19]{milnor}.

One can find disjoint open topological discs around points of
$P_f$ such that if $\M$ is the complement of the union of these discs
in the Riemann sphere, then $\M$ contains the Julia set of $f$,
$\M_1=f^{-1}(\M)\subset\M$, and there exists a
metric on $\M$ such that $f:\M_1\arr\M$ is strictly expanding. For
instance, one can take discs bounded by the equipotential lines of the
basins of attraction, see~\cite[Section~9]{milnor}.

Then $H_1(\M)$ is isomorphic to the quotient of the free group
$\Z^{|P_f|}$ generated by elements $a_z$, $z\in P_f$,
corresponding to the boundaries of the discs, modulo the relation
$\sum_{z\in P_f}a_z=0$. It is easy to see that the map
$\sigma=\iota_*\circ f^!$ acts by the rule 
\[\sigma(a_z)=\sum_{y\in f^{-1}(z)\cap P_f}a_y.\]

Suppose now that $\deg f$ is odd. We say that $z$ is a \emph{critical
value mod 2} if $|f^{-1}(z)|$ is even. Note that if $z$ is a
critical value mod 2, then it is a critical value, since then
$|f^{-1}(z)|\ne\deg f$. In particular, all critical values mod 2
belong to $P_f$. It is also easy to see that $z$ is a critical value
mod 2 if and only if the monodromy action of a small simple loop
around $z$ is an odd permutation. Namely, lengths of cycles of the
monodromy action are equal local degrees of $f$ in the preimages of
$z$. The sum of local degrees is equal to $\deg f$, i.e., is odd,
hence the number of odd local degrees is odd.
Parity of the monodromy action is equal
to parity of the number of cycles of even length, which is equal to
parity of $|f^{-1}(z)|$ minus the number of odd local degrees,
which is equal to parity of $|f^{-1}(z)|+1$.

\begin{proposition}
\label{pr:hyperbolicrational}
Let $f$ be a hyperbolic post-critically finite rational function. Let
$k$ be the number of attracting cycles of $f$. Let $l$ be the greatest
common divisor of their lengths.

If $\deg f$ is even, then
$\V_f/\V_f'$ is isomorphic to $\Z^{k-1}\oplus\Z/l\Z$.

If $\deg f$ is odd, and there exists an attracting cycle $C$ such that
the number of critical values mod 2 whose forward $f$-orbits are
attracted to $C$ is odd, then $\V_f/\V_f'$ is also isomorphic to
$\Z^{k-1}\oplus\Z/l\Z$. Otherwise, $\V_f/\V_f'\cong\Z/2\Z\oplus\Z^{k-1}\oplus\Z/l\Z$.
\end{proposition}

Note that $\Z^{k-1}\oplus\Z/l\Z$ coincides with the $K_1$-group of the
$C^*$-algebraic analog of $\V_f$, see~\cite{nek:cpalg}. Its
$K_0$-group $\Z/(d-1)\Z$ has also appeared in our paper, see Proposition~\ref{pr:VXorbits}.

\begin{proof}
Every attracting cycle of $f$ is superattracting (i.e., contains
critical points of $f$), hence it belongs to the post-critical set $P_f$.

Suppose at first that $\deg f$ is even.
If $z\in P_f$ does not belong to a cycle of $f:P_f\arr P_f$, then
there exists $n$ such that $f^{-n}(z)\cap P_f=\emptyset$ and hence
$\sigma^n(a_z)=0$. It follows that the images of such elements $a_z$
under the epimorphism $\pi:H_1(\M)\arr\V_f/\V_f'$ are equal to zero.

If $C$ is a cycle of $f:P_f\arr P_f$, then the images of $a_z$ in
$\V_f/\V_f'$, for $z\in C$, satisfy the relations $\pi(a_z)=\pi(a_{f(z)})$, since we
have $\sigma(a_{f(z)})=a_z$. It follows that $\V_f/\V_f'$ is the
quotient of $H_1(\M)$ by the relations making elements corresponding
to the points of each cycle of $f:P_f\arr P_f$ equal, and making equal
to zero all elements corresponding to elements of $P_f$ not belonging
to cycles. It follows that $\V_f/\V_f'$ is the quotient of the free
abelian group $\Z^k=\langle e_1, e_2, \ldots, e_k\rangle$ by the
relation $l_1e_1+l_2e_2+\cdots+l_ke_k=0$, where $l_i$ are the lengths
of the corresponding cycles of $f:P_f\arr P_f$. Consequently,
$\V_f/\V_f'\cong\Z^{k-1}\oplus\Z/l\Z$, where $l$ is the g.c.d.\ of
$l_1, l_2, \ldots, l_k$.

Suppose now that $\deg f$ is odd. Then $\V_f/\V_f'$ is isomorphic to
the quotient of $\Z/2\Z\oplus H_1(\M)$ by the relations
$\sigma_1(a)=a$, where $\sigma_1(t, g)=(t+\mathop{\mathrm{sign}}(g),
\sigma(g))$, where $\mathop{\mathrm{sign}}(g)$ is the parity of the
monodromy action of $g$ for the covering map $f:\M_1\arr\M$.

It follows that $\sigma_1$ acts on the elements of the form $(0,
a_z)$, for $z\in P_f$, by the rule
\[\sigma_1(0, a_z)=\left\{\begin{array}{ll} \left(1, \sum_{y\in
      f^{-1}(z)\cap P_f}a_y\right) & \text{if $z$ is a critical value mod
      2,}\\
 \left(0, \sum_{y\in
      f^{-1}(z)\cap P_f}a_y\right) &
  \text{otherwise.}\end{array}\right.\]

Suppose that $z\in P_f$ is such that no point of
$\bigcup_{n\ge 0}f^{-n}(z)$ is a critical value mod 2, and $z$ does not belong to a
cycle. Then there exists $n$ such that $\sigma_1^n(0, a_z)=0$, hence
image of $(0, a_z)$ in $\V_f/\V_f'$ is zero.

If $z\in P_f$ is a critical value mod 2, but no point
of $\bigcup_{n\ge 1}f^{-n}(z)$ is a critical value mod 2, then
$\sigma_1(0, a_z)=(1, \sigma(a_z))$, and hence the image of $(0, a_z)$ in $\V_f/\V_f'$
is equal to the image of $(1, 0)\in\Z/2\Z\oplus H_1(\M)$.

It follows by induction that if $z\in P_f$ does not belong to a cycle,
then the image of $(0, a_z)$ in $V_f/\V_f'$ is equal to the image of $(m,
0)\in\Z/2\Z\oplus H_1(\M)$, where $m$ is the parity of the number of
critical values mod 2 in the set $\bigcup_{n\ge 0}f^{-n}(z)$. In
particular, $\V_f/\V_f'$ is a quotient of $\Z/2\Z\oplus H$, where
$H\le H_1(\M)$ is the subgroup generated by $a_z$ for $z$ belonging to cycles of
$f:P_f\arr P_f$.

Suppose now that $C$ is a cycle of length $r$ of the map $f:P_f\arr
P_f$. For $z\in C$, denote by $z'$ the unique element of $C$ such that
$f(z')=z$, and by $t_z$ the parity of the number of
critical values mod 2 in the set
$B_z=\{z\}\cup\bigcup_{y\in f^{-1}(z)\setminus z'}\bigcup_{n\ge 0}f^{-1}(y)$, see Figure~\ref{fig:cycle}.
Then $\V_f/\V_f'$ is isomorphic to the quotient of $\Z/2\Z\oplus H$ by
the relations $(0, a_z)=(t_z, a_{z'})$.

\begin{figure}
\centering
\includegraphics{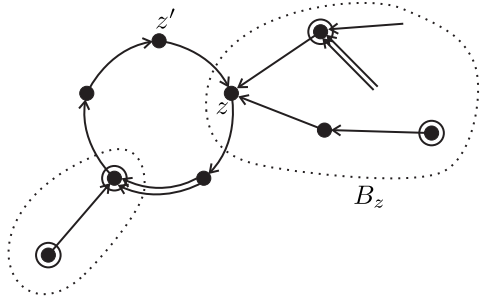}
\caption{A post-critical cycle}
\label{fig:cycle}
\end{figure}

Note that $t_C=\sum_{z\in C}t_z$ is the number of points $y$ that are
critical values mod 2 
and $f^n(y)\in C$ for all $n$ big enough. If $t_C$ is odd, then we
have a relation $(0, a_z)=(1, a_z)$, which implies that $(1, 0)$
belongs to the kernel of the epimorphism $\Z/2\Z\oplus
H_1(\M)\arr\V_f/\V_f'$, and the arguments for even $\deg f$ show that
$\V_f/\V_f'\cong\Z^{k-1}\oplus\Z/l\Z$. 

Suppose that $t_C$ is even for every cycle $C$ of $P_f$. Order
elements of every cycle $C\subset P_f$ into a sequence $z_0, z_1,
\ldots, z_{r-1}$ so that $f(z_i)=z_{i-1}$ for all $i=1, 2, \ldots,
r-1$, and $f(z_0)=z_{r-1}$. Denote then $b_{z_0}=(0, a_{z_0}),
b_{z_1}=(t_{z_0}, a_{z_1}), b_{z_2}=(t_{z_0}+t_{z_1}, a_{z_2}),
\ldots, b_{z_{r-1}}=(t_{z_0}+t_{z_1}+\cdots+t_{z_{r-2}},
a_{z_0})$. Then $\V_f/\V_f'$ is the quotient of $\Z/2\Z\oplus H$ by
the relations $b_{z_i}=b_{z_{i+1}}$ and $b_{z_{r-1}}=b_{z_0}$. It
follows that $\V_f/\V_f'$ is isomorphic to $\Z/2\Z\oplus\Z^{k-1}\oplus\Z/l\Z$.
\end{proof}

\begin{example}
If $f(z)=z^2+c$ is a hyperbolic post-critically finite quadratic
polynomial, then it has two
attracting cycles: $\{\infty\}$ and the orbit of the critical point
0. It follows that $\V_f/\V_f'$ is isomorphic to $\Z$.
\end{example}

\begin{example}
Suppose now that $f$ is a hyperbolic post-critically finite cubic
polynomial. If it has only one finite critical point $c$, then
$|f^{-1}(f(c))|=1$, hence there are no critical values mod 2. 

If $f$ has two critical points $c_1, c_2$, then $f(c_1)$ and $f(c_2)$
are critical values mod 2, and we have one of the following
possibilities:
\begin{itemize}
\item[(a)] forward orbits of $c_1$ and $c_2$ are disjoint cycles;
\item[(b)] both points $c_1$ and $c_2$ belong to a common cycle;
\item[(c)] one of the critical points belongs to a cycle $C$, and the
forward orbit of the other critical point eventually belongs to $C$.
\end{itemize}

It follows now from Proposition~\ref{pr:hyperbolicrational}
that $\V_f/\V_f'\cong\Z/2\Z\oplus\Z$ if the number of finite attracting
cycles of $f$ is 1, and $\V_f/\V_f'\cong\Z^2$ if it is 2.
\end{example}

\subsection{Finite presentation}
\begin{theorem}
\label{th:finitepresentation}
If $G$ is a contracting self-similar group, then the group $\V_G$ is
finitely presented.
\end{theorem}

\begin{corollary}
Let $f:\M\arr\M$ be an expanding self-covering of a compact path
connected metric space. Then the group $\V_f$ is finitely presented.
\end{corollary}

\begin{proof}
Let $\nuke$ be the nucleus of $G$. We may assume that $\nuke$ is a
generating set of $G$, since otherwise we can replace $G$ by
$\langle\nuke\rangle$ without changing $\V_G$.

For $v\in X^*$, and $g\in\V_G$,
denote by $L_v(g)$ the element of $\V_G$ defined by the rule
\[L_v(g)(w)=\left\{\begin{array}{cl} vg(u) & \text{if $w=vu$ for some
      $u\in X^\omega$}\\ w & \text{if $w$ does not start with
      $v$.}\end{array}\right.\]

The following is straightforward.

\begin{proposition}
For every $v\in X^*$ the map $L_v:\V_G\arr\V_G$ is a group
monomorphism. If $v, u\in X^*$ are not comparable, then the subgroups $L_v(\V_G)$
and $L_u(\V_G)$ of $\V_G$ commute. If $v, u\in X^*$ are non-empty, and
$h\in\V_G$ is such that $h(vw)=uw$ for all $w\in X^\omega$, then
$h\cdot L_v(g)\cdot h^{-1}=L_u(g)$ for all $g\in\V_G$.
\end{proposition}

Fix a letter $x_1\in X$. We will denote $L(g)=L_{x_1}(g)$.
For every pair $x, y\in X$ choose elements $A_{x, y}$ and $B_x$ of $\V_X$ such that
\[A_{x, y}(yw)=xyw,\quad B_x(x_1w)=xw,\]
for all $w\in X^\omega$. We assume that $B_{x_1}=1$.

Let $\langle S\;|\;R\rangle$ be a finite presentation of the
Higman-Thompson group $\V_X$, see~\cite{hgthomp}.
Let $S_1$ be the set of elements of $\V_G$ of the form $L(g)$ for
$g\in\nuke$.

\begin{lemma}
The set $S\cup S_1$ generates $\V_G$.
\end{lemma}

\begin{proof}
For every non-empty $v\in X^*$ we can find an element $h_v\in\V_X$ such that
$h_v(vw)=x_1w$ for all $w\in X^\omega$, see
Lemma~\ref{lem:incomplete}. Then $h_v^{-1}L(g)h_v=L_v(g)$
for all $g\in\V_G$. It follows that $L_v(g)\in\langle S\cup
S_1\rangle$ for all $g\in\nuke$ and $v\in X^*\setminus\{\varnothing\}$.
Every element $g\in\V_G$ can be
represented by a table $\left(\begin{array}{cccc} v_1 & v_2 & \ldots &
    v_n\\ g_1 & g_2 & \ldots & g_n\\ u_1 & u_2 & \ldots &
    u_n\end{array}\right)$, where $v_i, u_i$ are non-empty, and
$g_i\in\nuke$. But then
\[g=\left(\begin{array}{cccc} v_1 & v_2 & \ldots &
    v_n\\ 1 & 1 & \ldots & 1\\ u_1 & u_2 & \ldots &
    u_n\end{array}\right)L_{u_1}(g_1)L_{u_2}(g_2)\cdots
  L_{u_n}(g_n)\in\langle S\cup S_1\rangle.\]
\end{proof}

Represent each $A_{x, y}$, $B_x$ as group words $\overline A_{x, y}$,
$\overline B_x$ in $S$, and denote, for $y_1y_2\ldots y_n\in X^n$ and
$y\in X$,
\[\overline A_{y_1y_2\ldots y_n, y}=
\overline A_{y_1, y_2}\cdots\overline A_{y_{n-1}, y_n}\overline A_{y_n, y}.\]
Let $A_{v, y}$ be the image of $\overline A_{v, y}$ in $\V_X$.
Then $A_{v, y}$ satisfies
\[A_{v, y}(yw)=vyw\]
for all $w\in X^\omega$.

For every word $v=y_1y_2\ldots y_n$ of length at least 2 and every $g\in\V_G$ we have
\[L_v(g)= A_{y_1y_2\ldots y_{n-1}, y_n}B_{y_n} L(g)B_{y_n}^{-1}
A_{y_1y_2\ldots y_{n-1}, y_n}^{-1}.\]
For every $v=y_1y_2\ldots y_n\in X^*$ and $g\in\nuke$ denote by $\overline L_v(g)$ the
word
\[\overline A_{y_1y_2\ldots y_{n-1}, y_n}\overline B_{y_n}
L(g)\overline B_{y_n}^{-1}\overline A_{y_1y_2\ldots y_{n-1}, y_n}^{-1}\]
in generators $S\cup S_1$. Also denote by $\overline L_x(g)$ the
word $\overline B_xL(g)\overline B_x^{-1}$.

Denote by $\symm{d^n}$ the subgroup of $\V_X$ consisting of all
elements of the form $\sum_{i=1}^{d^n}S_{u_i}S_{v_i}^*$,
where $\{v_1, v_2, \ldots, v_{d^n}\}=\{u_1, u_2, \ldots, u_{d^n}\}=X^n$.
It is isomorphic to the symmetric group of degree $d^n$. Here $d=|X|$.

For every $x\in X$ choose a finite generating set $W_x$ (as a
set of group words in $S$) of the group $(\V_X)_{(X^\omega\setminus
xX^\omega)}$ of elements of $\V_X$ acting trivially on
$xX^\omega$. This group is isomorphic to the Higman-Thompson group
$G_{d, d-1}$, hence is finitely generated (see~\cite{hgthomp}).

Let $R_1$ be the union of the following sets of relations.
\begin{itemize}
\item[(\textbf{C})] \textbf{Commutation.} Relations of the
  form \[[\overline L_x(g_1), \overline L_y(g_2)]=[\overline L_{v_1}(g_1),
 \overline L_{v_2}(g_2)]=[L(g), h]=1\] for all $g, g_1, g_2\in\nuke,
 x, y \in X, v_1, v_2\in X^2, h\in W_{x_1}$, where $x\ne y$ and
 $v_1\ne v_2$.

\item[(\textbf{N})] \textbf{Nucleus.} Relations of the form \[L(g_1)L(g_2)L(g_3)=1\]
  for all $x\in X$ and $g_1, g_2, g_3\in\nuke$ such that $g_1g_2g_3=1$
  in $G$.

\item[(\textbf{S})] \textbf{Splitting.} Relations of the form
\[L(g)=\overline h\cdot\overline
  L_{x_1y_1}(g|_{y_1})\overline L_{x_1y_2}(g|_{y_2})\cdots\overline
  L_{x_1y_d}(g|_{y_d}),\] for all $g\in\nuke$, where $\overline h$ is a word in the generators
$S$ representing an element $h\in\symm{d^2}$ such that
$L(g)=h L_{x_1y_1}(g|_{y_1})L_{x_1y_2}(g|_{y_2})\cdots L_{x_1y_d}(g|_{y_d})$.
\end{itemize}

Let us show that 
the set $R\cup R_1$ is a set of defining relations for the group $\V_G$.
Denote by $\hat\V_G$ the group defined by the presentation $\langle
S\cup S_1\;|\;R\cup R_1\rangle$.
All relations $R\cup R_1$ hold in $\V_G$, hence $\V_G$ is a quotient
of $\hat\V_G$, and it is enough to show
that all relations in $\V_G$ also hold in $\hat\V_G$.

Note that since $R$ is a set of defining relations of $\V_X$, a group
word in $S$ is trivial in $\hat\V_G$ if
and only if it is trivial in $\V_X$. We will identify, therefore, the
elements of the subgroup $\langle S\rangle\le\hat\V_G$ with the corresponding elements of $\V_X$.

\begin{lemma}
\label{lem:action}
Suppose that $h\in\V_X$ and $u, v\in X^*$ are such that $h(uw)=vw$ for
all $w\in X^\omega$. Then $h\overline L_u(g)h^{-1}=\overline L_v(g)$
holds in $\hat\V_G$.
\end{lemma}

\begin{proof}
Let $u=a_1a_2\ldots a_n$ and $v=b_1b_2\ldots b_m$ for $a_i, b_i\in
X$. Then \[\overline L_u(g)=A_{a_1a_2\ldots a_{n-1},
  a_n}B_{a_n}L(g)B_{a_n}^{-1}A_{a_1a_2\ldots a_{n-1}, a_n}^{-1}\] and
\[\overline L_v(g)=A_{b_1b_2\ldots b_{m-1},
  b_m}B_{b_m}L(g)B_{b_m}^{-1}A_{b_1b_2\ldots b_{m-1}, b_m}^{-1},\] by
definition. We have dropped the lines above the letters $A$ and $B$,
because the corresponding elements belong to $\V_x$.

We have then
\begin{multline*}h\overline L_u(g)h^{-1}=
hA_{a_1a_2\ldots a_{n-1},
  a_n}B_{a_n}L(g)B_{a_n}^{-1}A_{a_1a_2\ldots a_{n-1},
  a_n}^{-1}h^{-1}=\\
A_{b_1b_2\ldots b_{m-1},  b_m}B_{b_m}\cdot 
B_{b_m}^{-1}A_{b_1b_2\ldots b_{m-1}, b_m}^{-1}hA_{a_1a_2\ldots a_{n-1}, a_n}B_{a_n}\cdot\\
L(g)\cdot\\
B_{a_n}^{-1}A_{a_1a_2\ldots a_{n-1}, a_n}^{-1}h^{-1}A_{b_1b_2\ldots b_{m-1}, b_m}B_{b_m}\cdot
B_{b_m}^{-1}A_{b_1b_2\ldots b_{m-1}, b_m}^{-1}.
\end{multline*}

The element 
\[f=B_{b_m}^{-1}A_{b_1b_2\ldots b_{m-1}, b_m}^{-1}hA_{a_1a_2\ldots a_{n-1}, a_n}B_{a_n}\]
satisfies
\[f(x_1w)=A_{v, y_m}^{-1}hA_{u, x_n}a_{y_m, x_n}(y_mw)=x_1w,\] for all
$w\in X^\omega$. Hence, by relations (\textbf{C}), $L(g)$ commutes with $f$,
i.e., $f\cdot L(g)\cdot f^{-1}=L(g)$ in
$\hat\V_G$, which finishes the proof.
\end{proof}

\begin{lemma}
\label{lem:Lcommutation}
If $v, u\in X^*$ are incomparable, then $\overline L_v(g_1)$ and
$\overline L_u(g_2)$ commute in $\hat\V_G$ for all $g_1, g_2\in\nuke$.
\end{lemma}

\begin{proof}
Let $x, y\in X$ be a pair of different letters. Then
$[\overline L_{xx}(g_1), \overline L_{xy}(g_2)]=1$ in
$\hat\V_G$, by (\textbf{C}). Since $v, u$ are incomparable,
either they both have length 1, or they form an incomplete
antichain. In the first case commutation of $\overline L_v(g_1)$ and
$\overline L_u(g_2)$ is a part of relations (\textbf{C}). In the second case, there exists
$a\in\V_X$ such that $a(uw)=xxw$ and $a(vw)=xyw$ for all
$w$ (see Lemma~\ref{lem:incomplete}). Then, by Lemma~\ref{lem:action},
\[[\overline L_v(g_1), \overline L_u(g_2)]=a[\overline
L_{xx}(g_1), \overline L_{xy}(g_2)]a^{-1}=1\]
in $\hat\V_G$.
\end{proof}

Let us prove now that any group word in $S\cup S_1$ that is trivial in
$\V_G$ is trivial in $\hat\V_G$. Note that relations (\textbf{S})
and Lemma~\ref{lem:action} imply relations
\begin{itemize}
\item[(\textbf{S'})] \[\overline L_v(g)=h\cdot\overline
  L_{vy_1}(g|_{y_1})\overline
  L_{vy_2}(g|_{y_2})\cdots\overline L_{vy_d}(g|_{y_d})\] for all
  $g\in\nuke$ and non-empty $v\in X^*$, where $h$ is an element of
  $\V_d$ such that $L_v(g)=hL_{vy_1}(g|_{y_1})L_{vy_2}(g|_{y_2})\cdots
  L_{vy_d}(g|_{y_d})$.
\end{itemize}

Every element of $\hat\V_G$ can be written in the form
$hL(g_1)^{h_1}L(g_2)^{h_2}\cdots L(g_n)^{h_n}$ for $h,
h_i\in\V_X$, and $g_i\in\nuke$.

Let $n_1$ be such that the element $h_n$ can be written as
$\sum_{i=1}^{d^{n_1}}S_{u_i}S_{v_i}^*$,
where $\{u_1, u_2, \ldots, u_{d^{n_1}}\}=X^{n_1}$. 
Using relations
(\textbf{S'}) and Lemma~\ref{lem:action}, we can rewrite the element
$L(g_n)$ in the form $\alpha\prod_{v\in X^{n_1-1}}\overline
L_{x_1v}(g_n|_v)$ for $\alpha\in\symm{d^{n_1}}$. (Note that
the factors 
$\overline L_{x_1v}(g_n|_v)$ commute with each other, by
Lemma~\ref{lem:Lcommutation}.)
Then for every $v\in X^{n_1-1}$ there exists $i$ such that $x_1v=u_i$,
and then by Lemma~\ref{lem:action}, we have
\[{\overline L_{x_1v}(g_n|_v)}^{h_n}=\overline L_{v_i}(g_n|_v),\]
so that $L(g_n)^{h_n}$ can be rewritten as a product of
$\alpha^{h_n}$ followed by a product of elements
of the form $\overline L_v(g_{v, n})$ for some $v\in X^{n_1}$ and
$g|_{v, n}\in\nuke$.

It follows by induction that every element of $\hat\V_G$ can be
written in the form
\begin{equation}
\label{eq:word}
g=h\overline L_{v_1}(g_1)\overline L_{v_2}(g_2)\cdots\overline
L_{v_m}(g_m)
\end{equation}
for some $v_i\in X^*$, $g_i\in\nuke$, and $h\in\V_X$.

Suppose that not all words $v_i$ are of the same length. Let $v_i$ be
the shortest, and let $k>|v_i|$ be the shortest length of words $v_j$
strictly longer than $v_i$. Using relations (\textbf{S'}), we can
rewrite $\overline L_{v_i}(g_i)$ as $\alpha_i\prod_{u\in
  X^{k-|v_i|}}L_{v_iu}(g_i|_u)$, and then, using
Lemma~\ref{lem:action}, move $\alpha_i\in\symm{d^k}$ to the beginning of the
product~\eqref{eq:word}. This procedure will increase the length of
the shortest word $v_i$ in~\eqref{eq:word} without changing the
length of the longest one. Repeating this procedure, we will
change~\eqref{eq:word} to a product of the same form, but in which all
words $v_i$ are of the same length.

Therefore, we may assume that in~\eqref{eq:word} all words $v_i$ are
of the same length $k$. Note that then $\overline
L_{v_1}(g_1)\overline L_{v_2}(g_2)\cdots\overline L_{v_m}(g_m)$ does
not change the beginning of length $k$ in any word $w\in
X^\omega$. Since $g$ is trivial in $\V_G$, this implies that $h$ does
not change the beginning of length $k$ in any word. It follows that we
can write $h$ as a product $\prod_{v\in X^k}L_v(h_v)$ for some
$h_v\in\V_X$. Using Lemmas~\ref{lem:action} and~\ref{lem:Lcommutation} we can
now rearrange the factors of~\eqref{eq:word} in such a way that
$g=\prod_{v\in X^k}f_v$, where $f_v=L_v(h_v)\overline L_v(g_{v,
  1})\overline L_v(g_{v, 2})\cdots\overline L_v(g_{v, m_v})$ for
$g_{v, i}\in\nuke$ and $h_v\in\V_X$. Note that $f_v$ are trivial in
$\V_G$. The latter implies that $h_v\in\symm{d^l}$ for some
$l$, and that the action of $h_vg_{v, 1}g_{v, 2}\cdots g_{v, m_v}$ on
$X^l$ is trivial. Consequently, using relations (\textbf{S'}), we can
rewrite $f_v$ as a product of elements of the form $\overline L_{vu}(g)$
for $u\in X^l$. Therefore, we may assume that $h_v$ are trivial. Then
$g_{v, 1}g_{v, 2}\cdots g_{v, m_v}$ is trivial in $G$. Relations
(\textbf{N}), (\textbf{S}), and
Propositions~\ref{pr:lengththree},~\ref{pr:kernelcontracting} finish the proof.
\end{proof}

\section{Dynamical systems and groupoids}

This section is an overview of relations between expanding
dynamical systems and self-similar groups, basic definitions of the
theory of \'etale groupoids, and properties of hyperbolic groupoids. 
For more details and proofs,
see~\cite{nek:book,nek:filling,nek:models} 
and~\cite{renault:groupoids,paterson:gr,haefl:foliations,nek:hyperbolic}.

\subsection{Limit dynamical system of a contracting group}

Let $(G, X)$ be a contracting self-similar group. Denote by
$X^{-\omega}$ the space of all left-infinite sequences $\ldots x_2x_1$
of elements of $X$ with the direct product topology. 

\begin{defi}
Sequences $\ldots x_2x_1, \ldots y_2y_1\in X^{-\omega}$ are
\emph{asymptotically equivalent} if there exists a finite set
$N\subset G$ and a sequence $g_n\in N$
such that \[g_n(x_n\ldots x_2x_1)=y_n\ldots
y_2y_1,\] 
for all $n\in\mathbb{N}$.
\end{defi}

Denote by $\J_G$ the quotient of the space $X^{-\omega}$ by the
asymptotic equivalence relation.

The asymptotic equivalence relation on $X^{-\omega}$ is
invariant with respect to the shift $\ldots x_2x_1\mapsto\ldots
x_3x_2$. Therefore, the shift
induces a continuous map $f:\J_G\arr\J_G$. The dynamical
system $(f, \J_G)$ is called the \emph{limit dynamical system} of $G$.

The map $f:\J_G\arr\J_G$ is expanding in the sense of
Definition~\ref{def:expandingmap} (even though it is not a covering
in general). Namely, we can represent $\J_G$ as the boundary of a
naturally defined Gromov hyperbolic graph (see~\cite{nek:hyplim}
and~\cite[Section~3.8]{nek:book}),
and some iteration
of $f$ will be locally uniformly expanding with respect to the visual
metric on the boundary.

\begin{defi}
\label{def:regular}
A self-similar group $(G, X)$ is said to be \emph{regular} if for every $g\in G$ there
exists a positive integer $n$ such that for every $v\in X^n$ either
$g(v)\ne v$, or $g|_v=1$.
\end{defi}

Note that it is enough to check the conditions of
Definition~\ref{def:regular} for elements $g$ of the nucleus of $G$.

The following proposition is proved in~\cite[Proposition~6.1]{nek:cpalg}.

\begin{proposition}
The shift map $f:\J_G\arr\J_G$ is a covering if and only if $G$ is regular.
\end{proposition}

\begin{defi}
\label{def:selfreplicating}
A self-similar action of $G$ on $X^*$ is said to be
\emph{self-replicating} (\emph{recurrent} in~\cite{nek:book})
if the left action of $G$ on the associated
biset is transitive, i.e., if for every $x, y\in X$ there exists $g\in
G$ such that $g\cdot x=y\cdot 1$.

An automorphism group $G$ of the rooted tree $X^*$ is said to be
\emph{level-transitive} if it is transitive on $X^n$ for every $n$.
\end{defi}

Note that every self-replicating group is level transitive. 

\begin{theorem}
Let $G$ be a contracting group. The space $\J_G$ is connected if and
only if $G$ is level-transitive. It is path connected if and only if
$G$ is self-replicating. If $G$ is self-replicating, then $\J_G$ is also
locally path connected.
\end{theorem}

\begin{proof}
Proof of connectedness, local connectedness, and path connectedness of
$\J_G$ under the appropriate conditions is given 
in~\cite[Theorem~3.6.3.]{nek:book}. 

If $\ldots x_2x_1$ and $\ldots y_2y_1$ are asymptotically equivalent
elements of $X^{-\omega}$, then for any $n$ the words $x_n\ldots
x_2x_1$ and $y_n\ldots y_2y_1$ belong to the same $G$-orbit. If there
exists $n$ such that the action of $G$ on $X^n$ is not transitive,
then partition of $X^n$ into $G$-orbits defines a partition of
$X^{-\omega}$ into clopen sets such that the asymptotic equivalence
relation identifies only points inside the sets of the partition. This
implies that $\J_G$ is disconnected.

The same arguments shows that if $G$ is not self-replicating, then
$\X_G$ is disconnected. Moreover, if $v_1, v_2\in X^n\cdot G$ belong
to different orbits of the left action, then for any $k, m\ge 0$, and
any $u_1, u_2\in X^k\cdot G$ and $w_1, w_2\in X^m\cdot G$ the elements
$u_1\otimes v_1\otimes w_1$ and $u_2\otimes v_2\otimes w_2$ belong to
different orbits of the left action. It follows that the set of
connected components of $\X_G$ is then uncountable. Consequently, the
set of path connected components of $\X_G$ is uncountable, and since
$G$ is countable, the set of path connected components of $\J_G=\X_G/G$ is
also uncountable.
\end{proof}

The following theorem is proved
in~\cite[Sections~5.3,~5.5]{nek:book} (in the context of length
metric spaces, but all the arguments remain to be valid in the general
case, if we use diameters of paths instead of their lengths, as in the
proof of Proposition~\ref{pr:expcontracting}.)

\begin{theorem}
Suppose that $f:\J\arr\J$ is an expanding self-covering of a path connected
space. Then $\img{f}$ is contracting, regular, self-replicating,
and the limit dynamical system of
$\img{f}$ is topologically conjugate to $(f, \J)$.

Let $G$ be a contracting regular self-replicating group. Then it is
equivalent, as a self-similar group (see
Definition~\ref{def:ssequivalent}), to the iterated monodromy group
of its limit dynamical system.
\end{theorem}

\begin{corollary}
Let $f_i:\J_i\arr\J_i$, for $i=1, 2$, be expanding self-coverings of
path connected compact spaces. Then $(f_1, \J_1)$ and $(f_2, \J_2)$
are topologically conjugate if and only if $\img{f_1}$ and $\img{f_2}$
are equivalent as self-similar groups.
\end{corollary}

\subsection{Limit solenoid and the limit $G$-space}
\label{ss:solenoid}

Let $X^{\Z}$ be the space of all bi-infinite sequences $\ldots
x_{-2}x_{-1}.x_0x_1\ldots$, where the dot denotes the place between the
coordinates number 0 and -1. 
Sequences $(x_n)_{n\in\Z}, (y_n)_{n\in\Z}\in X^\Z$ are
\emph{asymptotically equivalent} if there exists a finite set
$N\subset G$ and a sequence $g_n\in N$
such that
\[g_n(x_nx_{n+1}\ldots)=y_ny_{n+1}\ldots\] for all $n\in \Z$.

The quotient $\mathcal{S}_G$ of $X^\Z$ by the asymptotic equivalence
relation is called the \emph{limit solenoid} of the group $G$. The
shift $\ldots x_{-2}x_{-1}.x_0x_1\ldots\mapsto\ldots
x_{-3}x_{-2}.x_{-1}x_0\ldots$ induces a homeomorphism of
$\mathcal{S}_G$, which we will denote by $\hat f$.

It is shown in~\cite[Proposition~5.7.8.]{nek:book} that
the space $\mathcal{S}_G$ is naturally homeomorphic to the inverse
limit of the backward iterations of the limit dynamical system $f:\J_G\arr\J_G$:
\[\J_G\longleftarrow\J_G\longleftarrow\J_G\longleftarrow\cdots,\]
and the map $\hat f$ is conjugate to the map induced by $f$ on the inverse
limit. In other words, $(\hat f, \mathcal{S}_G)$ is the \emph{natural
extension} of the limit dynamical system $(f, \J_G)$. The
point of $\mathcal{S}_G$ represented by a sequence $\ldots
x_{-2}x_{-1}.x_0x_1\ldots\in X^{\Z}$ corresponds to the point of the
inverse limit represented by the sequence
\[\cdots\mapsto\ldots x_{-2}x_{-1}x_0x_1x_2\mapsto
\ldots x_{-2}x_{-1}x_0x_1\mapsto\ldots x_{-2}x_{-1}x_0\mapsto\ldots x_{-2}x_{-1}.\]

Another natural dynamical system associated with a contracting group $G$
is the limit $G$-space $\X_G$. Consider the topological space $X^{-\omega}\times
G$, where $G$ is discrete.
Two pairs $(\ldots x_2x_1, g), (\ldots y_2y_1, h)\in X^{-\omega}\times
G$ are \emph{asymptotically equivalent} if there exists a sequence
$g_n\in G$ taking a finite set of values such that for all $n\ge 1$
\[g_n\cdot x_n\ldots x_2x_1\cdot g=y_n\ldots y_2y_1\cdot h\]
in the $n$th tensor power $\Phi^{\otimes n}$ of the associated
$G$-biset, i.e., if
\[g_n(x_n\ldots x_2x_1)=y_n\ldots y_2y_1,\qquad g_n|_{x_n\ldots
  x_2x_1}g=h.\]
The quotient of $X^{-\omega}\times G$ by the asymptotic equivalence
relation is called the \emph{limit $G$-space} $\X_G$. We represent the points of the
space $\X_G$ by the sequences $\ldots x_2x_1\cdot g$, where $x_i\in X$
and $g\in G$.

The asymptotic equivalence relation on $X^{-\omega}\times G$ is
invariant with respect to the right action of $G$ on the second factor
of the direct product. It follows that this action induces a right action
of $G$ on $\X_G$ by homeomorphisms. The action of $G$ on $\X_G$ is
proper, and the space of orbits $\X_G/G$ is naturally homeomorphic to
$\J_G$.

The spaces $\J_G, \mathcal{S}_G, \X_G$ and the corresponding dynamical
systems depend only on the biset $\Phi$ associated with the
self-similar group. For example, $\X_G$ can be constructed in the
following way. 

Let $\Omega$ be the direct limit of the spaces $A^{-\omega}$,
where $A$ runs through all finite subsets of $\Phi$. We write a
sequence $(\ldots, a_2, a_1)\in A^{-\omega}$ as $\ldots\otimes
a_2\otimes a_1$. Two sequences $\ldots \otimes a_2\otimes a_1,
\ldots\otimes b_2\otimes b_1\in\Omega$ are said to be equivalent if
there exist a sequence $g_n\in G$ taking values in a finite set, such
that
\[g_n\cdot a_n\otimes \cdots\otimes a_2\otimes a_1=b_n\otimes
\cdots\otimes b_2\otimes b_1\]
in $\Phi^{\otimes n}$ for all $n$.

The quotient of $\Omega$ by this equivalence relation is naturally
homeomorphic to $\X_G$. Moreover, the homeomorphism conjugates the
natural action on $\X_G$ with the action induced by the natural right
action of $G$ on $\Omega$.

For every $v\cdot g\in X^n\cdot G=\Phi^{\otimes |v|}$ we have the map
$w\mapsto w\otimes (v\cdot g)$ on $\Omega$, given in terms of
$X^{-\omega}\times G$ by the rule
\[\ldots x_2x_1\cdot h\mapsto\ldots x_2x_1h(v)\cdot h|_vg.\]
It induces a continuous map $F_{v\cdot g}:\X_G\arr\X_G$. If $G$ is
regular, then $F_{v\cdot g}$ is a covering map.

Since the limit dynamical system 
$f:\J_G\arr\J_G$ is induced by the shift on $X^{-\omega}$, the maps
$F_{v\cdot g}:\X_G\arr\X_G$ are lifts of $f^{-|v|}$ by the quotient map
$P:\X_G\arr\X_G/G=\J_G$, i.e., we have equality $f^{|v|}\circ P\circ F_{v\cdot g}=P$.

\subsection{Groupoids of germs}
\begin{defi}
Let $\X$ be a topological space. A \emph{pseudogroup} acting on $\X$
is a set $\wt\G$ of homeomorphisms between open subset of $\X$ that is
closed under
\begin{enumerate}
\item \emph{compositions}: if $F_1:U_1\arr V_1$ and $F_2:U_2\arr V_2$ are
  elements of $\wt\G$, then $F_1\circ F_2:F_2^{-1}(V_2\cap
  U_1)\arr F_1(V_2\cap U_1)$ is an element of $\wt\G$;
\item \emph{taking inverse}: if $F:U\arr V$ is an element of
  $\wt\G$, then $F^{-1}:V\arr U$ is an element of $\wt\G$;
\item \emph{restriction} to an open subset: if $F:V\arr U$ is an element of
  $\wt\G$ and $V'$ is an open subset of $V$, then $F|_{V'}\in\wt\G$;
\item \emph{unions}: if for a homeomorphism $F:U\arr V$ between open subsets
  of $\X$ there exists a covering $U_i$ of $U$ by open subsets, such
  that $F|_{U_i}\in\wt\G$ for all $i$, then $F\in\wt\G$.
\end{enumerate}
We always assume that the identical homeomorphism $\X\arr\X$ belongs
to the pseudogroup.
\end{defi}

Let $\wt\G$ be a pseudogroup acting on $\X$. A \emph{germ}
of $\wt\G$ is equivalence class of 
a pair $(F, x)$, where $F\in\wt\G$, and $x$ is a point of
the domain of $F$. Two pairs $(F_1, x_1)$ and $(F_2, x_2)$ represent
the same germ (are equivalent) if and only if $x_1=x_2$ and there
exists a neighborhood $U$ of $x_1$ such that $F_1|_U=F_2|_U$.

The set of all germs of $\wt\G$ has a natural topology whose
basis consists of sets of the form $\{(F,
x)\;:\;x\in\mathop{\mathrm{dom}}(F)\}$, where $F\in\wt\G$.

If $(F_1, x_1)$ and $(F_2, x_2)$ are such germs that $F_2(x_2)=x_1$,
then we can compose them:
\[(F_1, x_1)(F_2, x_2)=(F_1\circ F_2, x_2).\]
\emph{Inverse} of a germ $(F, x)$ is the germ $(F^{-1}, F(x))$. The
set $\G$ of all germs of $\wt\G$ is a groupoid with respect to these
operations (i.e., a small category of isomorphisms).

The groupoid $\G$ is \emph{topological}, i.e., the
operations of composition and taking inverse
are continuous.

\begin{example}
If $f:\X\arr\X$ is a covering map, then restrictions of $f$ to open
subsets $U\subset\X$ such that $f:U\arr f(U)$ is a homeomorphism
generate a pseudogroup. Its groupoid of germs $\mathfrak{F}$ will be called
\emph{groupoid of germs generated by $f$}. Every element of
$\mathfrak{F}$ can be represented as a product $(f^n, x)^{-1}(f^m, y)$
for some $x, y\in\X$ such that $f^m(y)=f^n(x)$.
\end{example}

\begin{example}
If $G$ is a group acting on a topological space $\X$, then every germ
of the pseudogroup generated by $G$ is a germ of an element of
$G$. Therefore, the \emph{groupoid of germs of $G$} is the set of
equivalence classes of pairs $(g, x)\in G\times\X$, where $(g_1, x_1)$
and $(g_2, x_2)$ are equivalent if and only if $x_1=x_2$, and
$g_1^{-1}g_2$ fixes all points of a neighborhood of $x_1$. This
groupoid is in general different from the \emph{groupoid of the
action}, which is equal as a set to $G\times\X$.
\end{example}

If $\G$ is a groupoid of germs of a pseudogroup acting on a space
$\X$, then we identify the germ $(1, x)$ of the identical
homeomorphism $1:\X\arr\X$ with the point $x$ of $\X$, and call
elements of the form $(1, x)$ the \emph{units} of the groupoid.
We will sometimes denote $\X$ by $\G^{(0)}$, as the space of units of $\G$.

For $(F, x)\in\G$, we denote by $\be(F, x)=x$ and $\en(F, x)=F(x)$ the
\emph{origin} and \emph{target} of the germ. Two germs $g_1, g_2\in\G$
are \emph{composable} (i.e., $g_1g_2$ is defined) if and only if
$\en(g_2)=\be(g_1)$.

We say that points $x, y\in\G^{(0)}$ \emph{belong to one orbit} if
there exists $g\in\G$ such that $x=\be(g)$ and $y=\en(g)$. This is an
equivalence relation on $\G^{(0)}=\X$, and this notion of orbits
coincides with the natural notion of orbits of pseudogroups.
A set $A\subset\G^{(0)}$ is a \emph{$\G$-transversal} if it intersects
every $\G$-orbit.

If $A$ is a subset of $\G^{(0)}$, then \emph{restriction} $\G|_A$ of $\G$ to
$A$ is the groupoid of all elements $g\in\G$ such that $\be(g),
\en(g)\in A$. The \emph{isotropy} group $\G_x$, for $x\in\G^{(0)}$,
is the group of elements $g\in\G$ such that $\be(g)=\en(g)=x$.

Note that the pseudogroup $\wt\G$ can be reconstructed from the
groupoid of its germs $\G$. Namely, a \emph{bisection} is a
subset $F\subset\G$ of the groupoid, such that
$\be:F\arr\be(F)$ and $\en:F\arr\en(F)$ are homeomorphisms. Every open
bisection $F$ defines a homeomorphism $\be(F)\arr\en(F)$ by the rule
$\be(g)\mapsto\en(g)$ for $g\in F$. The set $\wt\G$ of all open
bisections is a pseudogroup, and if $\G$ is the groupoid of germs of
a pseudogroup, then the pseudogroup of bisections coincides with $\wt\G$. We say
that $\wt\G$ is the \emph{associated pseudogroup} of the groupoid.

\begin{defi}
\label{def:equivalentgroupoids}
Let $\G_1, \G_2$ be groupoids of germs. We say that they are
\emph{equivalent} if there exists a groupoid $\G$ such that $\G^{(0)}$
is the disjoint union $\G_1^{(0)}\sqcup\G_2^{(0)}$, restrictions of
$\G$ to $\G_i^{(0)}$ is equal to $\G_i$ for every $i=1, 2$, and the sets
$\G_i^{(0)}$ are $\G$-transversals.
\end{defi}

The following procedure is a standard way of constructing a groupoid equivalent to a given
one. Namely, let
$p:\mathcal{Y}\arr\G^{(0)}$ be a local homeomorphism, i.e., for every
$y\in\mathcal{Y}$ there exists a neighborhood $U$ of
$y$ such that $p:U\arr p(U)$ is a homeomorphism. Suppose that
$p(\mathcal{Y})$ is a $\G$-transversal. Then \emph{lift} of $\G$ by $p$ is
the groupoid of germs of the pseudogroup generated by all local
homeomorphisms of the form $p'\circ F\circ p:U\arr W$, where 
\begin{itemize}
\item $U$ is such that $p:U\arr p(U)$ is a homeomorphism, 
\item $p(U)$ is contained in the domain of $F$, 
\item $W$ is such that $p:W\arr F(p(U))$ is a homeomorphism, 
\item $p'$ is the inverse of $p:W\arr F(p(U))$.
\end{itemize}
Then the map $p$ induces a morphism from the lift of $\G$ to $\G$,
mapping the germ of $p'\circ F\circ p$ at $x$ to the germ of $F$ at $p(x)$.

\begin{example}
Consider the \emph{trivial groupoid} on a manifold $\M$, i.e., the
groupoid consisting of units only. Let $\pi:\widetilde\M\arr\M$ be the
universal covering. Then lift of the trivial groupoid by $\pi$ is the
groupoid of germs of the action of the fundamental group on
$\widetilde\M$. (In this case it coincides with the groupoid of the action.)
\end{example}

\begin{defi}
Let $\G$ be a groupoid of germs. It is said to be \emph{proper} if the
map $(\be, \en):\G\arr\G^{(0)}\times\G^{(0)}$ is proper, i.e., if
preimages of compact subsets of $\G^{(0)}\times\G^{(0)}$ under this
map are compact.
\end{defi}

The groupoid $\G$ is proper if and only if for every compact subset
$C$ of $\G^{(0)}$ the set of elements $g\in\G$ such that $\be(g),
\en(g)\in C$ is compact.

If $\G$ is proper, then for every $x\in\G^{(0)}$ the isotropy
group $\G_x$ is finite.

Every groupoid equivalent to a proper groupoid is proper. If $\G$ is
proper, then the space of orbits of $\G$ is Hausdorff.

Let $\G$ be a groupoid of germs. Its \emph{topological full group}
$[[\G]]$ is the set of all bisections $F$ such that
$\be(F)=\en(F)=\G^{(0)}$, i.e., the set 
of homeomorphisms $F:\G^{(0)}\arr\G^{(0)}$ such that all germs
of $F$ belong to $\G$. See~\cite{gior:full}, where the notion of a topological
full group (for a groupoid of germs generated by one homeomorphism)
was introduced.

\begin{example}
Let $f:\M_1\arr\M$ be a partial self-covering. Then $\V_f$ is the full
topological group of the groupoid of germs of the local homeomorphisms
$S_\gamma$ of the boundary of the tree $T_t$ for $t\in\M$.
\end{example}

\begin{example}
Let $G$ be a self-similar group acting on $X^\omega$. Let $\G$ be the
groupoid of germs of the pseudogroup generated by the action of $G$
and the germs of the homeomorphisms $S_x(x_1x_2\ldots)=xx_1x_2\ldots$
for $x\in X$. It is easy to see that the topological full group of
$\G$ is the group $\V_G$.
\end{example}

\subsection{Hyperbolic groupoids}

Here we present a very short overview of the basic definitions and
results of the paper~\cite{nek:hyperbolic}.

Let $\G$ be a groupoid of germs. A \emph{compact generating pair} of $\G$ is a
pair $(S, \X_1)$, where $S\subset\G$ and $\X_1\subset\G^{(0)}$ are
compact, $\X_1$ contains an open $\G$-transversal,
$S\subset\G|_{\X_1}$, and for every
$g\in\G|_{\X_1}$ there exists $n$ such that $(S\cup S^{-1})^n$ is a
neighborhood of $g$ in $\G|_{\X_1}$.

A groupoid is \emph{compactly generated} if it has a compact
generating pair. See a variant of this definition
in~\cite{haefliger:compactgen}. A groupoid equivalent to a compactly generated
groupoid is also compactly generated. 

Let $(S, \X_1)$ be a compact generating pair of $\G$. Let
$x\in\X_1$. Then the \emph{Cayley graph} $\Gamma(x, S)$ is the
oriented graph with the set of vertices \[\{g\in\G\;:\;\be(g)=x,
\en(g)\in\X_1\},\] in which there is an arrow from $g_1$ to $g_2$ if
and only if $g_2g_1^{-1}\in S$.

A vertex path $v_1, v_2, \ldots$ 
in a graph $\Gamma$ (i.e., a sequences of vertices such that $v_i$ is
adjacent to $v_{i+1}$) is said to be a $C$-quasi-geodesic (where $C>1$ is
a constant) if $|v_i-v_j|\ge C^{-1}|i-j|+C$ for all $i, j$.

\begin{defi}
\label{def:hyperbolic}
A Hausdorff groupoid of germs $\G$ is \emph{hyperbolic} if there is
a compact generating pair $(S, \X_1)$ of $\G$, a metric $|x-y|$
defined on a neighborhood of $\X_1$, and constants $L, C>1, \Delta>0$
such that
\begin{enumerate}
\item Every element $g\in S$ is a germ of a homeomorphism $F\in\wt\G$ such
  that $|F(x)-F(y)|\le L^{-1}|x-y|$ for all $x, y\in\dom F$.
\item For every $x\in\X_1$ the Cayley graph $\Gamma(x, S)$ is Gromov
  $\Delta$-hyperbolic.
\item For every $x\in\X_1$ there exists a point $\omega_x$ of the
  boundary of $\Gamma(x, S)$ such that every oriented path in the
  Cayley graph $\Gamma(x, S^{-1})$ is a $C$-quasi-geodesic converging
  to $\omega_x$.
\item $\be(S)=\en(S)=\X_1$.
\item All elements of the pseudogroup $\wt\G$.
\end{enumerate}
\end{defi}

\begin{example}
Let $f:\J\arr\J$ be an expanding self-covering of a compact metric space. Then the groupoid of
germs generated by $f$ is hyperbolic. The corresponding generating
pair is $(S, \J)$, where $S$ is the set of germs of $f^{-1}$. The
corresponding Cayley graphs $\Gamma(x, S)$ are trees. The special
point $\omega_x$ of the boundary is the limit of the forward germs
$(f^n, x)$ for $n\to+\infty$. See more in~\cite[Section~5.2]{nek:hyperbolic}
\end{example}

Let $\G$ be a hyperbolic groupoid. For every $x\in\G^{(0)}$ there
exists a generating pair $(S, \X_1)$ satisfying the conditions of
Definition~\ref{def:hyperbolic}
and such that $x\in\X_1$. Denote by $\partial\G_x$ the boundary of the
Cayley graph $\Gamma(x, S)$ minus the point $\omega_x$. The space
$\partial\G_x$ does not depend on the generating pair.

Let $(S, \X_1)$ be a generating pair satisfying the conditions of
Definition~\ref{def:hyperbolic}. Find a finite set of contractions
$\mathcal{F}\subset\tilde\G$ such that $S\subset\bigcup_{F\in\mathcal{F}} F$, i.e.,
every element $s\in S$ is a germ of a contraction $F\in\mathcal{F}$. Every point
$\xi\in\partial\G_x$ can be represented as the limit of a sequence
vertices of the Cayley graph of
the form $g, s_1g, s_2s_1g, \ldots, s_n\cdots s_2s_1g, \ldots$, where
$g\in\G$, and $s_i\in S$.
There exists $\epsilon>0$ (not depending on $\xi$) and a sequence
$F_i\in\mathcal{F}$ such that $s_i$ is equal to a germ $(F_i, x_i)$, and the
$\epsilon$-neighborhood of $x_i=\be(s_i)$ belongs to the domain of
$F_i$. Then there exists $\delta$ (also depending only on $S$ and
$\mathcal{F}$) such that the $\delta$-neighborhood of $\en(g)$ belongs to
the domain of $F_n\circ\cdots\circ F_2\circ F_1$ for all
$n$.

Let $F\in\wt\G$ be such that $g\in F$ and $\en(F)$ belongs to the
$\delta$-neighborhood of $\en(g)$. Then $\be(F_n\circ\cdots\circ
F_2\circ F_1\circ F)=\be(F)$ for every $n$.

It is shown in~\cite{nek:hyperbolic} that there exists a topology on
the disjoint union $\partial\G$ of the spaces $\partial\G_x$,
$x\in\G^{(0)}$, which agrees with the topology on its subsets
$\partial\G_x$ and such that the map
\begin{equation}
\label{eq:locprod}
(y, \xi)\mapsto \lim_{n\to\infty}(F_n\circ
F_{n-1}\circ\cdots\circ F_1\circ F, y)\in\partial\G_y
\end{equation}
is a well defined homeomorphism (if $U=\be(F)$ is small enough)
from the direct product of $U=\be(F)$
with a subset of $\partial\G_x$ to a subset of $\partial\G$, see
Figure~\ref{fig:tube}.

\begin{figure}
\centering
\includegraphics{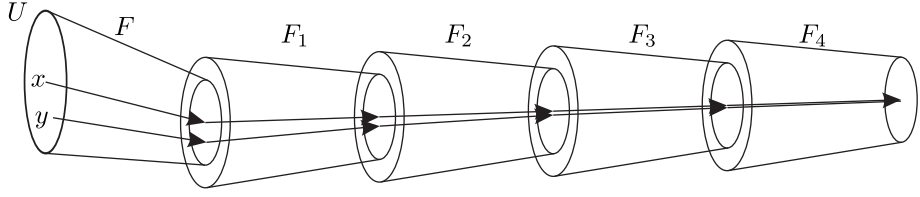}
\caption{Composition of contractions}
\label{fig:tube}
\end{figure}

Moreover, these homeomorphisms agree with a natural \emph{local product
structure} of $\partial\G$. Namely, a basis of the topology on
$\partial\G$ consists of \emph{rectangles}, i.e., sets with a
decomposition into a direct product of topological spaces, such that
the decompositions agree where they overlap, and locally coincide with
the maps given by~\eqref{eq:locprod}.

The groupoid $\G$ acts on the space $\partial\G$ from the
right. Namely, every $g\in\G$ defines a natural homeomorphism
$\partial\G_{\en(g)}\arr\partial\G_{\be(g)}$ mapping the limit of a
sequence $g_n\in\Gamma(\en(g), S)$ to the limit of the sequence
$g_ng\in\Gamma(\be(g), S)$. This action is an action of the topological
groupoid $\G$ on the topological space $\partial\G$ over the projection map
$P:\partial\G\arr\G^{(0)}$ mapping all points of $\partial\G_x$ to
$x$.

The action of $\G$ on $\partial\G$ (i.e., the associated action of
$\tilde\G$ on $\partial\G$ by local homeomorphisms) 
preserves with the local product structure of $\partial\G$. Naturally
defined projection of the action of $\G$ onto the first coordinate of
the local product decomposition is equivalent to $\G$, while the projection
onto the second coordinate is the \emph{dual groupoid} of $\G$.

Let us give an equivalent, and maybe more intuitive, definition of the dual groupoid.

\begin{defi}
Let $\overline{\Gamma(x, S)}$ and $\overline{\Gamma(y, S)}$ be the
Cayley graphs of $\G$ with adjoined boundaries $\partial\G_x$ and
$\partial\G_y$. A homeomorphism $F:U\arr V$ between open neighborhoods
$U\subset\overline{\Gamma(x, S)}$ and $V\subset\overline{\Gamma(y,
  S)}$ of points of $\partial\G_x$ and $\partial\G_y$ is an
\emph{asymptotic morphism} if for every sequence of pairwise different
edges $(g_1, h_1), (g_2, h_2), \ldots$ in $U$ the distance between
$g_ih_i^{-1}$ and $F(g_i)F(h_i)^{-1}$ goes to zero.
\end{defi}

Note that $g_ih_i^{-1}$ and $F(g_i)F(h_i)^{-1}$ belong to a compact
subset of $\G$, hence the notion of convergence of their distance to
zero does not depend on the choice of a metric on $\G$.

\begin{defi}
The groupoid $\bdry\G$ of germs of restrictions of the asymptotic morphisms 
to the spaces $\partial\G_x$, $x\in\G^{(0)}$, is
the \emph{dual groupoid} of $\G$.
\end{defi}

The space of units of $\bdry\G$ is the topologically disjoint union of
the spaces $\partial\G_x$. (In particular, it is not separable.)
If $\G$ is minimal (i.e., if all orbits are dense),
then $\partial\G_x$ is an open transversal of the
dual groupoid for any $x\in\G^{(0)}$,
hence the dual groupoid can be defined as the groupoid
of germs at $\partial\G_x$ of the asymptotic morphisms.
We will denote it $\bdry\G_x$.

We will denote by $\G^\top$ any groupoid equivalent to $\bdry\G$.
The following theorem is proved in~\cite{nek:hyperbolic}.

\begin{theorem}
Let $\G$ be a minimal Hausdorff hyperbolic groupoid. Then the dual
groupoid $\G^\top$ is minimal, Hausdorff, and hyperbolic, and
$(\G^\top)^\top$ is equivalent to $\G$.
\end{theorem}

\subsection{Groupoid of germs generated by an expanding self-covering}

Let $f:\J\arr\J$ be an expanding self-covering of a path connected
compact metric space. Denote by $\mathfrak{F}$ the groupoid of germs
generated by $f$. Every element of $\mathfrak{F}$ can be written as a
product $(f^n, x)^{-1}(f^m, y)$, for $n, m\in\N$, and $x, y\in\J$ such
that $f^n(x)=f^m(y)$.

A \emph{natural extension} of $f$ is the inverse limit $\hat\J$ of the
maps $f$ together with the homeomorphism $\hat f$ of $\hat\J$ induced
by $f$, see~\ref{ss:solenoid}. Let $P_S:\hat\J\arr\J$ be the natural projection.

For every point $x\in\J$ there exists a neighborhood $U$ that is
evenly covered by each map $f^n:\J\arr\J$, since $f$ is expanding. It
follows that the set $P^{-1}(U)$ is naturally decomposed into the
direct product of $U$ with the boundary $\partial T_x$ of the tree of
preimages of any point $x\in U$. 

The groupoid $\mathfrak{F}$ is hyperbolic, and we can
consider the space $\partial\mathfrak{F}$ together with the projection
$P:\partial\mathfrak{F}\arr\mathfrak{F}^{(0)}=\J$.
Let us use the generating set $S$ of
$\mathfrak{F}$ equal to the set of germs of the inverse map
$f^{-1}$. Then the Cayley graphs $\Gamma(x, S)$ are regular trees such
that every vertex has one incoming and $d=\deg f$ outgoing arrows. The
fiber $\partial\mathfrak{F}_x$ is equal to the boundary of this tree
minus the limit of the path $(f^n, x)$, $n\ge 0$. In other words, it
is the natural inductive limit of the boundaries of the preimage
trees $T_{f^n(x)}$ for $n\ge 0$.

Let $\coc:\mathfrak{F}\arr\Z$ be the homomorphism (\emph{cocycle})
given by the rule $\coc(f, x)=-1$, so that $\coc((f^n, x)^{-1}(f^m,
y))=n-m$. See Figure~\ref{fig:cayleye}, where the Cayley graph $\Gamma(x, S)$
together with the levels of the cocycle $\coc$ are shown.

\begin{figure}
\centering
\includegraphics{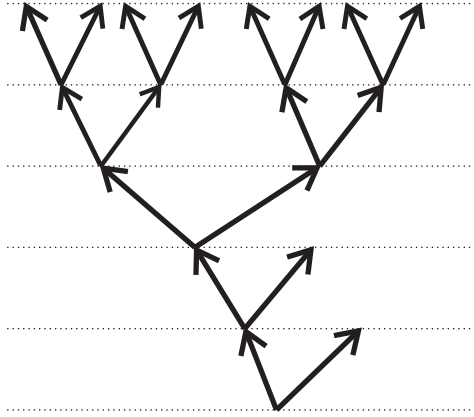}
\caption{Cayley graph of $\mathfrak{F}$}
\label{fig:cayleye}
\end{figure}

Every point of $\partial\mathfrak{F}_x$ can be uniquely represented as the
limit of a sequence $s_n\cdots s_2s_1\cdot g$, where $s_i\in S$, and
$g\in\mathfrak{F}$ is such that $\be(g)=x$ and $\coc(g)=0$. Note that
the set of limits of the sequences $s_n\cdots s_2s_1\cdot g$ for a
fixed $g$ and all possible choices of $s_i\in S$ is
naturally identified with the fiber $P_S^{-1}(\en(g))$ of the solenoid
$\hat\J$ (i.e., with $\partial T_{\en(g)}$). It
follows then directly from the definitions that the space
$\partial\mathfrak{F}$ is homeomorphic to the subset 
$\{(\zeta, g)\;:\;P_S(\zeta)=\en(g)\}$
of the direct product $\hat\J\times\mathfrak{F}_0$, 
where $\mathfrak{F}_0$ is the subgroupoid
$\coc^{-1}(0)\subset\mathfrak{F}$. 
The action of
$\mathfrak{F}$ on $\partial\mathfrak{F}$ is given in these terms by the rules
\[(\zeta, g)\cdot h=\left\{\begin{array}{ll}({\hat f}^n(\zeta),
    f^n\circ gh), & \text{if $\coc(h)=n>0$,}\\
({\hat f}^{-n}(\zeta), s_n\cdots s_2s_1gh), & \text{if $\coc(h)=-n<0$,}\\
(\zeta, gh), &\text{if $\coc(h)=0$,}
\end{array}\right.\]
where $s_i\in S$ are such that $\zeta=\lim_{m\to\infty}s_m\cdots
s_2s_1$.

We get hence the following description of the natural extension $\hat
f:\hat\J\arr\hat\J$ in terms of the $\mathfrak{F}$-space
$\partial\mathfrak{F}$.

\begin{proposition}
\label{pr:solenoidflow}
The quotient of the space $\partial\mathfrak{F}$ by the action of
$\mathfrak{F}_0=\coc^{-1}(0)$ is homeomorphic to $\hat\J$. If
$F\in\tilde{\mathfrak{F}}$ is such that $\coc(F)=\{n\}$, then the
germs of the map induced by $F$ on the quotient space
$\hat\J=\partial\mathfrak{F}/\mathfrak{F}_0$ are germs of the map
${\hat f}^{-n}$.
\end{proposition}

We say that points $\xi, \zeta\in\hat\J$ are
\emph{unstably equivalent} if distance between
$\hat f^{-n}(\xi)$ and $\hat f^{-n}(\zeta)$ goes to zero as $n\to
+\infty$. They are said to be \emph{stably equivalent} if distance between
$\hat f^n(\xi)$ and $\hat f^n(\zeta)$ goes to zero as $n\to +\infty$.

A \emph{leaf} of $\hat\J$ is its path connected component. (Recall
that we assume that $\J$ is path connected.) Every leaf 
is an equivalence class of the unstable
equivalence relation on $\hat\J$. 

Each leaf is dense in $\hat\J$, and it is more natural to consider it
with the \emph{inductive limit topology}.
Namely, a subset $U$ of a leaf $\mathcal{L}$ is open if and only if
its intersection with every compact subset $C\subset\mathcal{L}$ is
open in $C$. Note that for every compact set $C\subset\mathcal{L}$ the
map $P:C\arr\J$ is finite-to-one.

Restriction of the map
$P:\hat\J\arr\J$ to any leaf $\mathcal{L}$ of $\hat\J$ is a covering
map. Let $\mathfrak{F}_{\mathcal{L}}$ be the lift of the groupoid
$\mathfrak{F}$ to the leaf $\mathcal{L}$ by this covering. Then the
groupoid $\mathfrak{F}_{\mathcal{L}}$ is equivalent to
$\mathfrak{F}$. The cocycle $\coc$ lifts to the cocycle $\coc\circ P$ on
$\mathfrak{F}_{\mathcal{L}}$, which we will denote by
$\coc_{\mathcal{L}}$ or just $\coc$.

It follows then from the definition of $\partial\G$ for a hyperbolic
groupoid $\G$, that 
the space $\partial\mathfrak{F}_{\mathcal{L}}$ is the fiber product of
the maps $P_S:\mathcal{L}\arr\J$ and $P:\partial\mathfrak{F}\arr\J$,
i.e., the subset $\{(x, y)\;:\;P_S(x)=P(y)\}$ of
$\mathcal{L}\times\partial\mathfrak{F}$. We get the following
corollary of Proposition~\ref{pr:solenoidflow}.

\begin{corollary}
\label{cor:solenoidflow}
Let $\mathcal{L}$ be a leaf of $\hat\J$, and let
$\mathfrak{F}_{\mathcal{L}}$ be the lift of $\mathfrak{F}$ by the
covering $P_S:\mathcal{L}\arr\J$.
Then the quotient of the space $\partial\mathfrak{F}_{\mathcal{L}}$ by the action of
$\mathfrak{F}_{\mathcal{L}, 0}=\coc_{\mathcal{L}}^{-1}(0)$ is
homeomorphic to $\hat\J$. If $F\in\tilde{\mathfrak{F}_{\mathcal{L}}}$
is such that $\coc_{\mathcal{L}}(F)=\{n\}$, then the
germs of the map induced by $F$ on the quotient space
$\hat\J=\partial\mathfrak{F}_{\mathcal{L}}/\mathfrak{F}_{\mathcal{L},
0}$ are germs of the map ${\hat f}^{-n}$.
\end{corollary}

Let $G$ be a contracting regular self-replicating group. Let $\G$ be
the groupoid of
germs of the action of the group $\V_G$ on $X^\omega$. It is generated
by the groupoid of germs of $G$ and the germs of the maps
$S_x:x_1x_2\ldots\mapsto xx_1x_2\ldots$. It is shown
in~\cite[Subsection~5.3.]{nek:hyperbolic} that $\G$
is hyperbolic, and its dual is the groupoid generated by the limit
dynamical system $f:\J_G\arr\J_G$.

More explicitly, let $w\in X^\omega$. Then the boundary
$\partial\G_w$ is the leaf of the limit solenoid $\mathcal{S}_G$
consisting of points representable by sequences $\ldots
x_{-2}x_{-1}\cdot x_0x_1\ldots$, where $x_0x_1\ldots$ belongs to the
$G$-orbit of $w$. The groupoid $\bdry\G_w$ is equal to the
lift of the groupoid generated by the limit dynamical system
$f:\J_G\arr\J_G$ to the leaf $\partial\G_w$ (by the covering induced by the 
projection $\mathcal{S}_G\arr\J_G$ of the natural extension onto $\J_G$).

\section{Reconstruction of the dynamical system from $\V_f$}

The main result of this section is the following classification of the
groups $\V_f$.

\begin{theorem}
\label{th:classification}
Let $f_i:\J_i\arr\J_i$, for $i=1, 2$, be expanding self-coverings of
path connected compact metric spaces. Then $\V_{f_1}$ and $\V_{f_2}$
are isomorphic as abstract groups if and only if the dynamical systems
$(f_1, \J_1)$ and $(f_2, \J_2)$ are topologically conjugate.
\end{theorem}

\begin{example}
One can show that the limit dynamical systems of two groups
$\mathfrak{K}_{v_i}$, $i=1, 2$, are topologically conjugate if and
only if either $v_1=v_2$, or $v_1$ can be obtained from $v_2$ by
replacing each 0 by 1 and each 1 by 0. Namely, the sequence of letters
of $v_i$ can be interpreted as a \emph{kneading sequence} of the
dynamical system, which in turn can be defined in purely topological
terms. This gives a complete classification of the groups
$\V_{\mathfrak{K}_v}$ up to isomorphism.
\end{example}

\subsection{M.~Rubin's theorem}

Recall that if $G$ is a group acting on a topological space $\X$, and
$U\subseteq\X$ is an open subset, then we denote by $G_{(U)}$ the
group of elements $g\in G$ acting trivially outside of $U$, see
Subsection~\ref{ss:somegeneralfacts}.

The following theorem is proved in~\cite[Theorem~0.2]{rubin:reconstr}.

\begin{theorem}
\label{th:rubin}
Let $G_i$, for $i=1, 2$, be groups acting faithfully by homeomorphisms on Hausdorff
topological spaces $\X_i$. Suppose that the following conditions hold
for both pairs $(G, \X)=(G_i, \X_i)$, $i=1, 2$.
\begin{enumerate}
\item For every non-empty open subset $U\subset\X$ the group $G_{(U)}$
  is non-trivial.
\item For every non-empty open subset $U\subset\X$ there exists a
  non-empty open subset $U_1\subseteq U$ such that if $V, W\subset
  U_1$ are open sets such that there exists $g\in G$ such that
  $g(V)\cap W\ne\emptyset$, then there exists $g\in G_{(U)}$ such that
  $g(V)\cap W\ne\emptyset$.
\end{enumerate}
Then for every isomorphism $\phi:G_1\arr G_2$ there exists a
homeomorphism $F:\X_1\arr\X_2$ inducing it, i.e., such that
$\phi(g)=F\circ g\circ F^{-1}$ for all $g\in G$.
\end{theorem}

We say that a group $G$ acting on a topological space $\X$
is \emph{locally transitive} if there exists a basis of open sets
$\mathcal{U}$ such that for every $U\in\mathcal{U}$ the
group $G_{(U)}$ has a dense orbit in $U$.

The following is a direct corollary of Theorem~\ref{th:rubin}.

\begin{corollary}
\label{cor:rubin}
If $G_i$ are locally transitive groups of homeomorphisms of
topological spaces $\X_i$, then every isomorphism $\phi:G_1\arr G_2$
is induced by a homeomorphism $F:\X_1\arr\X_2$.
\end{corollary}

Similar results (with simpler proofs), which can be applied to many
groups $\V_G$, are proved in~\cite{gior:full,medynets:reconstruction,matui:fullonesided}.

It is easy to see that the Higman-Thompson group $\V_{|X|}$ acting on the space
$X^\omega$ is locally transitive. It follows that every group of
homeomorphisms of $X^\omega$ containing the Higman-Thompson group is
locally transitive, which implies the following fact.

\begin{theorem}
\label{th:vfrubin}
Let $G_i$ be groups acting on the Cantor sets $X_i^\omega$ and
containing the Higman-Thompson groups $\V_{|X_i|}$. Then every
isomorphism $\phi:G_1\arr G_2$ is induced by a homeomorphism
$F:X_1^\omega\arr X_2^\omega$.
\end{theorem}

\subsection{Proof of Theorem~\ref{th:classification}}

If $f_1:\J_1\arr\J_1$ and $f_2:\J_2\arr\J_2$ are topologically
conjugate self-coverings of path-connected spaces, then the groups
$\V_{f_1}$ and $\V_{f_2}$ are obviously isomorphic, since they were
defined in purely topological terms.

Let us prove the converse implication for expanding maps. 
By Theorem~\ref{th:vfrubin},
if groups $\V_{f_1}$ and $\V_{f_2}$ are isomorphic, then their action
on the corresponding spaces $X_i^\omega$ are topologically conjugate,
hence the groupoid of germs of the action of $\V_{f_i}$ on $X_i^\omega$
are isomorphic.

Therefore, it is enough to show that if $f:\J\arr\J$ is an
expanding self-covering of a compact path-connected metric space, then
the dynamical system $(f, \J)$ can be reconstructed from the
topological groupoid $\G$ of germs of the action of $\V_f$ on
$X^\omega$.

Denote by $\mathfrak{F}$ the groupoid of germs generated by
$f:\J\arr\J$. We identify $\J$ with the limit space $\J_G$ of the
self-similar group $G=\img{f}$, and hence encode points of $\J$ by
sequences $\ldots x_2x_1\in X^{-\omega}$. Recall that $f$ acts then by
the shift $\ldots x_2x_1\mapsto\ldots x_3x_2$. Let
$\coc:\mathfrak{F}\arr\Z$ be the cocycle (groupoid homomorphism)
defined by the condition that $\coc(f, x)=-1$ for all $x\in\J$.

The groupoids $\mathfrak{F}$ and $\G$ are hyperbolic and mutually
dual. Let $w\in X^\omega$ be an arbitrary point, and
denote $\gH=\bdry\G_w$ and $\cH=\partial\G_w=\gH^{(0)}$. It is enough
to show that $(f, \J)$ is uniquely determined (up to a topological
conjugacy) by the groupoid $\gH$.

Denote by $\Omega_w$ the set of bi-infinite sequences $\ldots
x_{-2}x_{-1}.x_0x_1\ldots$ such that $x_0x_1\ldots$ belongs to the
$G$-orbit of $w$. Note that $\J$ is path connected,
$G$ is self-replicating, hence $G$-orbits coincide with the
$\V_f=\V_G$-orbits. We consider $\Omega_w$ with the topology of the disjoint
union of the set of the form $X^{-\omega}.x_0x_1\ldots$.

Then the space $\cH=\partial\G_w$ is naturally identified with the quotient of the
space $\Omega_w$ by the asymptotic equivalence relation (defined in
the same way as on $X^{\Z}$, see Subsection~\ref{ss:solenoid}). Let
$P_S:\cH\arr\J$ be the natural
projection induced by $\ldots x_{-2}x_{-1}.x_0x_1\ldots\mapsto\ldots
x_{-2}x_{-1}$. It is a covering map, and $\gH$ is the
lift of $\mathfrak{F}$ by $P_S$.
We will also denote by $P_S$ the corresponding functor (homomorphism of groupoids)
$P_S:\gH\arr\mathfrak{F}$. 

Let us show at first that the cocycle $\coc:\gH\arr\Z$ (equal to the
lift of the cocycle $\coc:\mathfrak{F}\arr\Z$) is
uniquely determined by the structure of the topological groupoid
$\gH$.

\begin{proposition}
\label{pr:concomp}
Let $\mathcal{C}$ be a connected component of $\gH$. Then
$\be:\mathcal{C}\arr\cH$, $\en:\mathcal{C}\arr\cH$
are coverings. 

If $\coc(\mathcal{C})\ne 0$, then $\mathcal{C}$
contains a non-trivial element of infinite order in the isotropy group $\gH_x$
of a point.

If $\coc(\mathcal{C})=0$, then the groupoid generated by $\mathcal{C}$
is proper.
\end{proposition}

\begin{proof}
Let $\X_G$ be the limit $G$-space. The action of $G$ on
$\X_G$ is free, and the maps $F_v:\xi\mapsto\xi\otimes v$ are coverings
for all $v\in X^n\cdot G$.

For $w\in X^\omega$, the leaf $\partial\G_w=\cH$ is the image of $\X_G$
under the map $P_w:\xi\mapsto \xi\cdot w$.
This map coincides with the quotient of $\X_G$ by the action of the
stabilizer $G_w$.

Let $\mathfrak{X}$ be the groupoid of germs with the space of units
$\X_G$ generated by the germs of the action of $G$ and the germs of
the maps $F_v(\xi)=\xi\otimes v$ for $v\in X^*\cdot G$. Then
$\mathfrak{X}$ is the lift
of $\gH$ by the quotient map $P_w:\X_G\arr\cH$.

Every element of $\gH$ is a germ of the transformation 
\[F_{v\cdot g, u\cdot h}:\xi\otimes v.g(w)\mapsto\xi\otimes u.h(w),\]
for some $g, h\in G$ and $u, v\in X^*$. 

The germ $(F_{v\cdot g, u\cdot h}, \zeta\otimes v.g(w))$
can be lifted to the germ $(\tilde F_{v\cdot g, u\cdot h},
\zeta\otimes v\cdot g)$ of
the local homeomorphism
\[\tilde F_{v\cdot g, u\cdot h}:\xi\otimes v\cdot g\mapsto\xi\otimes u\cdot h\]
of $\mathfrak{X}$. It follows that every element of $\gH$ is
a germ of $P_wF_{u\cdot h}F_{v\cdot g}^{-1}P_w^{-1}$. The
space $\X_G$ is connected, the maps $F_{u\cdot h}, F_{v\cdot g},
P_w$ are coverings, hence if $\mathcal{C}$ is the connected component of
the germ $(F_{v\cdot g, u\cdot h}, \zeta\otimes v.g(w))$, then
$\be:\mathcal{C}\arr\cH$ and $\en:\mathcal{C}\arr\cH$ are covering maps.

Suppose that $\coc(\mathcal{C})\ne 0$. It means that every germ
$(F_{v\cdot g, u\cdot h}, \zeta\otimes v.g(w))\in\mathcal{C}$ is such that $|v|\ne
|u|$. Without loss of generality, we may assume that $|u|>|v|$. Let
$u=u_1v_1$, where $|v_1|=|v|$. Since $G$ is self-replicating, there
exists $g_1\in G$ such that $g_1\cdot v\cdot g=v_1\cdot h$ in the
biset. Then a lift of the germ $(F_{v\cdot g, u\cdot h}, \zeta\otimes
v.g(w))$ to $\mathfrak{X}$ is a germ of the
transformation
\[\xi\otimes v\cdot
g\mapsto \xi\otimes u_1\otimes v_1\cdot h=\xi\otimes u_1\cdot
g_1\otimes v\cdot g.\] The point $\zeta=\ldots u_1\cdot g_1\otimes u_1\cdot
g_1\otimes u_1\cdot g_1\in\X_G$ is well defined (as it is the image of
a point of $\Omega$, see Subsection~\ref{ss:solenoid}), and it satisfies
$\zeta=\zeta\otimes u_1\cdot g_1$. Then the germ of the transformation
\[\xi\otimes v\cdot g\mapsto \xi\otimes u_1\cdot g\otimes v\cdot g\]
at $\zeta\otimes v\cdot g$ is a non-trivial
contracting element of the isotropy group of $\zeta\otimes v\cdot
g$. It is contained in the connected component of the germs of the
transformation
$\xi\otimes v\cdot g\mapsto \xi\otimes u\cdot h$. Mapping everything to
$\gH$ by $P_w$, we find a non-trivial contracting (hence infinite
order) element of an isotropy group.

If $\coc(\mathcal{C})=0$, then elements of $\mathcal{C}$ are germs of
transformations of the form $\xi\otimes
v.g(w)\mapsto\xi\otimes u.h(w)$, where $g, h\in G$ and $v, u\in X^*$
are such that $|v|=|u|$. There exists $g_1\in G$ such that $u\cdot
h=g_1\cdot v\cdot g$ in $X^n\cdot G$. It follows that elements of
$\mathcal{C}$ are lifted by $P_wF_{v\cdot g}:\mathfrak{X}\arr\gH$ to the
action of $g_1$ on $\X_G$. It follows that the groupoid generated by
$\mathcal{C}$ lifts by $P_wF_{v\cdot g}$ to a subgroupoid of the action of $G$ on $\X_G$,
and hence is proper.
\end{proof}

\begin{proposition}
\label{pr:cocycleunique}
The cocycle $\coc:\gH\arr\Z$ is uniquely determined by the topological
groupoid $\gH$.
\end{proposition}

\begin{proof}
It follows from~\ref{pr:concomp} that 
the value of $\coc$ on a connected component
$\mathcal{C}$ of $\gH$ is zero if and only if $\mathcal{C}$
generates a proper groupoid. (Since isotropy groups of a proper
groupoid are finite.)

Let $g_1, g_2$ be arbitrary elements of $\gH$. By the first claim of
Proposition~\ref{pr:concomp}, there exist $g_1', g_2'$ in the components of $g_1$
and $g_2$, respectively, such that the product $g_1'g_2'$ is
defined. Note that the connected component of $g_1'g_2'$ depends only
on the connected components of $g_1$ and $g_2$. It follows that the
set of connected components of $\gH$ is a group. The
quotient of this group by the subgroup of components on which $\coc$
is zero is isomorphic to $\Z$. Since the set of components on which
$\coc$ is zero is uniquely determined by the topological groupoid, we
conclude that the set $\{\coc, -\coc\}$ is uniquely determined by the
structure of the topological groupoid $\gH$. But we can distinguish
between $\coc$ and $-\coc$ using~\cite[Proposition~3.4.1.]{nek:hyperbolic}.
\end{proof}

The next statement follows now directly from
Proposition~\ref{pr:cocycleunique} and Corollary~\ref{cor:solenoidflow}.

\begin{proposition}
The natural extension $\hat f:\hat\J\arr\hat\J$ is uniquely
determined, up to topological conjugacy, by the groupoid $\G$.
\end{proposition}

Suppose that $f_1:\J_1\arr\J_1$ and $f_2:\J_2\arr\J_2$ are two
expanding homeomorphisms with the same natural extension
$\hat f:\mathcal{S}\arr\mathcal{S}$. It remains to prove that $(f_1, \J_1)$
and $(f_2, \J_2)$ are topologically conjugate.

Denote by 
$P_i:\mathcal{S}\arr\J_i$ the corresponding projections.
Let $\tilde\J$ be the image of $\mathcal{S}$ in $\J_1\times\J_2$ under the
map $(P_1, P_2)$. It is compact and connected, since so is
$\mathcal{S}$. We will denote by
$\tilde P_i:\tilde\J\arr\J_i$ the restrictions of the projections
$\J_1\times\J_2\arr\J_i$.

Since $P_i$ locally are projections on the unstable coordinate of
the local product decomposition of $\mathcal{S}$ (which depends only
on the conjugacy class of $(\hat f, \mathcal{S})$), for every
$\xi\in\mathcal{S}$ there exists a rectangular neighborhood $U\ni\xi$
such that $P_i:U\arr\pi_i(U)$ is decomposed into the composition of
projection of $U$ onto its unstable direction and a homeomorphism of
this direction with $P_i(U)$. Moreover, since $\mathcal{S}$ is
compact, we can cover $\mathcal{S}$ by a finite number of such
rectangles $U$. 

The map $\hat f:\mathcal{S}\arr\mathcal{S}$ induces a map
$\tilde f:\tilde\J\arr\tilde\J$ by the rule $\tilde f(\xi_1, \xi_2)=(f_1(\xi_1),
f_2(\xi_2))$. The projections $\tilde P_i$ are semi-conjugacies of $\tilde f$ with
$f_i$.

Let $(\xi_1, \xi_2)\in\tilde\J$, i.e., there exists
$\xi\in\mathcal{S}$ such that $\xi_i=P_i(\xi)$. There exists a
rectangular neighborhood $U$ of $\xi$ such that $P_i$ are
projections onto the unstable direction composed with a
homeomorphism, and the unstable direction of $U$ is connected.
If $U$ is small enough, then $(\hat f)^{-1}(U)$ is decomposed
into a union of a finite set $\mathcal{R}$ of rectangles on which each of $P_i$
is a homeomorphism with projection onto the unstable direction.
Consider the sets $(P_1, P_2)(R)$ for $R\in\mathcal{R}$.
We get a finite number of components of $(\tilde
f)^{-1}((P_1, P_2)(U))$ such that $\tilde f$ is a homeomorphism on
each of them. It follows that $\tilde f$ is a finite degree covering map.

For every $\xi\in\J_1$ the set $P_1^{-1}(\xi)$ is a compact subset of
$\mathcal{S}$ contained in one stable equivalence class. Consequently,
there exists $n_0$ such that $P_2(\hat f^n(P_1^{-1}(\xi)))$ is a
single point for all $n\ge n_0$. It follows that there exists a small
neighborhood $U$ of $\xi$ such that the map $P_2\circ \hat f^{n_0}\circ
P_1^{-1}=f_2^{n_0}\circ P_2\circ P_1^{-1}$ is a homeomorphism on $U$.
By compactness, there exists
$n_1$ such that $P_2\circ f^{n_1}\circ P_1^{-1}=f_2^{n_1}\circ 
P_2\circ P_1^{-1}$ is a well defined covering map from $\J_1$ to $\J_2$.

It follows that the projections $\tilde P_i:\tilde\J\arr\J_i$ are finite degree
covering maps.  For every point $t^{(j)}_i\in
{\tilde P}_i^{-1}(t_i)$ we have the corresponding tree $T_{t^{(j)}_i}$
of preimages under iterations of $\tilde f$. They
are disjoint (more precisely, for every $n$ the sets ${\tilde f}^{-n}(t^{(j)})$
are disjoint for different $t^{(j)}$).

By the arguments above, there exists
$n_1$  such that $\tilde P_i(z_1)=\tilde P_i(z_2)$ implies
${\tilde f}^{n_1}(z_1)={\tilde f}^{n_1}(z_2)$.
But this contradicts the fact that the trees
$T_j$ are disjoint. It follows that $\tilde P_i$ have degree 1, i.e., are
homeomorphisms conjugating $f$ with $f_i$.

\subsection{Equivalence of groupoids}

\begin{theorem}
\label{th:groupoidequivalence}
Let $f_i:\J_i\arr\J_i$, for $i=1, 2$, be expanding self-coverings of
connected and locally connected compact metric spaces. Then the
following conditions are equivalent.
\begin{enumerate}
\item The dynamical systems $(f_1, \J_1)$ and $(f_2, \J_2)$ are
  topologically conjugate.
\item The groupoids generated by germs of $f_1$ and $f_2$ are equivalent.
\item The natural extensions of $f_1$ and $f_2$ are topologically
  conjugate.
\item The natural extensions of $f_1$ and $f_2$ generate equivalent
  groupoids of germs.
\item The actions of $\V_{f_1}$ and $\V_{f_2}$ on the corresponding
  Cantor sets are topologically conjugate.
\item The groupoids of germs generated by the actions of $\V_{f_1}$
  and $\V_{f_2}$ on the corresponding Cantor sets are equivalent.
\item The self-similar groups $\img{f_1}$ and $\img{f_2}$ are
  equivalent.
\item The groups $\V_{f_1}$ and $\V_{f_2}$ are isomorphic as abstract
  groups.
\end{enumerate}
\end{theorem}

\begin{proof}
The groupoid of germs $\mathfrak{F}_i$ generated by $f_i$, the groupoid of germs
$\G_i$ generated by $\V_{f_i}$, and the groupoid of germs generated by the
natural extension uniquely determine each other, up to equivalence of
groupoids, since the first two are mutually dual hyperbolic groupoids,
and the third one is their geodesic flow, see~\cite{nek:hyperbolic}.
Equivalence of (7) and (1) is proved in~\cite{nek:book}.

It remains, therefore, to prove that the equivalence class of
$\mathfrak{F}_i$ uniquely determines $(f_i, \J_i)$.

Suppose that $w_1$ and $w_2$ belong to one orbit of the groupoid $\G$
from Definition~\ref{def:equivalentgroupoids}.
Let $g\in\G$ be such that $\be(g)=w_2$ and
$\en(g)=w_1$. Then the map $h\mapsto hg$ is a quasi-isometry between the
Cayley graphs of $\G_1$ and $\G_2$ based at $w_1$ and $w_2$ respectively,
inducing an isomorphism $\bdry\G_{w_1}\arr\bdry\G_{w_2}$. We have
shown during the proof of Theorem~\ref{th:classification}
that the dynamical systems $(f_i, \J_i)$ can be uniquely
reconstructed from the topological groupoids $\bdry\G_{w_i}$, which
implies that $(f_1, \J_1)$ and $(f_2, \J_2)$ are topologically conjugate.
\end{proof}

\def\cprime{$'$}
\providecommand{\bysame}{\leavevmode\hbox to3em{\hrulefill}\thinspace}
\providecommand{\MR}{\relax\ifhmode\unskip\space\fi MR }
\providecommand{\MRhref}[2]{%
  \href{http://www.ams.org/mathscinet-getitem?mr=#1}{#2}
}
\providecommand{\href}[2]{#2}

\end{document}